\documentclass[twoside,11pt,a4paper,leqno]{amsart}

\usepackage[latin1]{inputenc}
\usepackage{amsfonts}
\usepackage{amssymb}
\usepackage{amsmath}
\usepackage{amsthm}
\usepackage{cite}
\usepackage{graphicx}
\usepackage{amscd}
\usepackage{color}
\usepackage{bm}
\usepackage{enumerate}
\usepackage[dvips]{epsfig}
\usepackage{psfrag}
\usepackage{epsfig}
\usepackage{xcolor}

\usepackage[abbrev,backrefs]{amsrefs}
\allowdisplaybreaks

\usepackage{cite}
\newtheorem{theorem}{Theorem}
\newtheorem{definition}[theorem]{Definition}
\newtheorem{corollary}[theorem]{Corollary}
\newtheorem{proposition}[theorem]{Proposition}
\newtheorem{lemma}[theorem]{Lemma}
\newtheorem{remark}[theorem]{Remark}

\numberwithin{theorem}{section}

\numberwithin{equation}{section}

\newcommand{\abs}[1]{\lvert#1\rvert}
\newcommand{\norm}[1]{\lVert#1\rVert}

\newcommand{\R}{\mathbb{R}}

\newcommand{\N}{\mathbb{N}}

\renewcommand{\S}{\mathcal{S}}

\newcommand{\cB}{{\mathcal B}}

\newcommand{\cM}{{\mathcal M}}

\newcommand{\weak}{\rightharpoonup}

\newcommand{\B}{{\bf B}}

\newcommand{\eps}{\varepsilon}

\renewcommand{\phi}{\varphi}

\renewcommand{\epsilon}{\varepsilon}

\let\psfragfont\footnotesize
\let\psfragfonta\tiny

\psfrag{Oe}{\psfragfont$\Omega_\eps$}
\psfrag{z0}{\psfragfont$\tilde z_0$}
\psfrag{Ke}{\psfragfont$K_\eps$}
\psfrag{A}{\psfragfont$A_\eps$}
\psfrag{twoeps}{\psfragfont$2\eps$}
\psfrag{0}{\psfragfont$0$}
\psfrag{Y0}{\psfragfont$Y_0$}
\psfrag{rn-1}{\psfragfont$\R^{N-1}$}
\psfrag{rn+}{\psfragfont$\R^N_+$}
\psfrag{or}{\psfragfonta$0$}
\psfrag{u}{\psfragfont$u$}
\psfrag{u?}{\psfragfont$u$ ?}
\psfrag{e1}{\psfragfonta$e_1$}
\psfrag{Al}{\psfragfonta$A_\lambda?$}
\psfrag{-e1}{\psfragfonta$-e_N$}
\psfrag{B}{\psfragfont$\B$}
\psfrag{k}{\psfragfonta$K$}
\psfrag{eh}{\psfragfonta$e_H$}
\psfrag{alh}{\psfragfonta$\alpha_H$}
\psfrag{sigma}{\psfragfonta$\Sigma$}
\psfrag{om}{\psfragfonta$\Omega$}
\psfrag{om1}{\psfragfonta$\Omega_1$}
\psfrag{om2}{\psfragfonta$\Omega_2$}
\psfrag{dh}{\psfragfonta$T_\lambda$}
\psfrag{ph}{\psfragfonta$T_\lambda$}
\psfrag{h}{\psfragfonta$K_\lambda$}
\psfrag{x}{\psfragfont$x$}
\psfrag{z}{\psfragfont$z$}
\psfrag{xl}{\psfragfont$x^\lambda$}
\psfrag{xk}{\psfragfonta$y_k$}
\psfrag{ex}{\psfragfont$e$}
\psfrag{eh}{\psfragfont$e_H$}
\psfrag{y}{\psfragfont$y$}
\psfrag{o}{\psfragfont$0$}
\psfrag{x1}{\psfragfont$\overline x$}
\psfrag{x_1}{\psfragfont$x_1$}
\psfrag{y1}{\psfragfont$\overline y$}
\psfrag{yl}{\psfragfont$y^\lambda$}
\psfrag{lam}{\psfragfont$ $}
\psfrag{b}{\psfragfont$\B$}
\psfrag{pb}{\psfragfont$\partial \B$}
\psfrag{sr}{\psfragfont$S_r$}

\begin{document}
	
	\title[A sharp threshold in dimension $N$]{A sharp threshold for Trudinger-Moser type  inequalities with logarithmic kernels in dimension $N$}

	\author{Alessandro Cannone and Silvia Cingolani}
	\address{A. Cannone \hfill\break
		Dipartimento di Matematica,  Universit\`{a} degli Studi di Bari Aldo Moro \hfill\break Via Orabona 4, 70125  Bari, Italy}
	\email{alessandro.cannone@uniba.it}
\address{S. Cingolani \hfill\break
	Dipartimento di Matematica,  Universit\`{a} degli Studi di Bari Aldo Moro \hfill\break Via Orabona 4, 70125  Bari, Italy}
\email{silvia.cingolani@uniba.it}

	\date{}
	
	\keywords{Trudinger-Moser inequality, extremal functions, threshold, logarithmic kernel, $N$ dimension}
	
	\subjclass[2000]{35J60, 35Q40, 31B10, 35B33}

%	\maketitle
%	\tableofcontents

%	\maketitle
%	\tableofcontents
	
\begin{abstract}	
In the paper we investigate  Trudinger-Moser type inequalities in presence of logarithmic kernels in dimension $N$.
A sharp threshold, depending on $N$, is detected for the existence of extremal functions or  blow-up, where the domain is the ball or the entire space $\R^N$. 
We also show that the extremal functions satisfy suitable Euler-Lagrange equations. 
When the domain is the entire space, such  equations can be derived by a $N$-Laplacian Schr\"odinger equation strongly coupled with a higher order fractional Poisson's equation.
The results extends \cite{CiWe2} to any dimension $N \geq 2$.
\end{abstract}

%	\maketitle
%	\tableofcontents

\maketitle

\tableofcontents

\section{Introduction}\label{sec:introduction}

\bigskip

Within PDE theory, the class of nonlocal interactions equations, involving logarithmic kernels, has gained great interest in the last decades. 
Such nonlocal problems have arisen by applications in different contexts, ranging from vortex theory,
statistical dynamics of selfgravitating clouds \cite{Suzuki,W},  quantum theory for crystals 
\cite{Dolbeault-Perthame},  to the description of vortices in turbulent Euler flows \cite{CLMP}.

The first results were based on numerical methods. For a long time the rigorous analitycal study remained a quite open field, due to the difficulties to handle sign-changing kernels,  not bounded from above nor from below. 
Recently variational results have been employed to treat integro-differential equations, involving logarithmic convolution potentials, where classical PDE theory fails (see  \cite{BCS,BVS,CiJe,CiWe,CiWeYu,DuWe,Masaki,Masaki2,stubbe,SV}).

Recently, in \cite{CiWe2},  planar logarithmic Trudinger-Moser type inequalities  
and characterizations of the critical nonlinear growth rates for these inequalities are  established.

In this paper, we wish to extend the results in 
\cite{CiWe2}  to any dimension $N \geq 2$.
This problem contains some new and delicate aspects related to the quasilinear context in which the maximization will set up.

%Precisely we consider the nonlocal interaction functional of the form 
%\begin{equation}
 % \label{eq:logarithmic-convolution-kernel}
%u \mapsto \int_{\Omega} \int_{\Omega} \ln \frac{1}{|x-y|}G(u(x))G(u(y))\,dx dy 
%\end{equation} 
%where $\Omega= B_1$ is the unit ball of $\R^N$ or $\Omega = \R^N $.

% and $F: \R \to \R$ is a nonlinearity which can grow exponentially.
% We notice that under suitable restrictions on the growth of $F$, one may combine (\ref{dis}) with the logarithmic Hardy-Littlewood-Sobolev inequality due to Beckner %\cite{beckner} to derive a logarithmic convolution inequality of Moser-Trudinger type. However, as we shall discuss in detail in Remark~\ref{sec:remark-} below, such a %combination does not allow for sharp critical growth rates which we wish to consider in the present paper.

\smallskip
Firstly we consider the unit ball 
$$
\cB_1:= \{u \in W^{1,N}_0(B_1)\::\: |\nabla u|_N \le 1 \}
$$
where $|\cdot |_N $ is the standard Lebesgue norm  of $L^N(\R^N)$, $N \geq 2$.

We aim to maximize  the quantity
\begin{equation}\label{rottt}
	m_1(G,N) := \sup_{u \in \cB_1} \Phi(u) 
\end{equation}
where 
\begin{equation}  \label{eq:logarithmic-convolution-kernel}
\Phi(u)= \int_{B_1} \int_{B_1} \ln \frac{1}{|x-y|}G(u(x))G(u(y))\,dx dy, 
\end{equation} 
and  $G$ satisfies the assumption
\begin{itemize}
\item[$(G_0)$] $G: \R \to [0,\infty)$ is even, continuous, and  strictly increasing on $[0,\infty)$. Moreover, there exist constants $\alpha,c>0$ with 
	\begin{equation}
		\label{eq:general-growth-condition}
		G(t) \le c e^{\alpha |t|^{N/(N-1)}}  \qquad \text{for $t \in \R$.}
	\end{equation}
\end{itemize}

Since the logarithmic kernel function in (\ref{eq:logarithmic-convolution-kernel}) changes sign, it is not a priori clear that the double integral in (\ref{eq:logarithmic-convolution-kernel}) has a well defined value. This can be proved, splitting  the kernel $\ln \frac{1}{|\,\cdot\,|}$ into its positive and negative part and seeing $\Phi$ as the difference of two finite functionals $\Phi_\pm$.

\bigskip
In order to study the maximization problem  we stress that  the value 
 $\alpha_N:= N \omega_{N-1}^{1/(N-1)}$, where $\omega_{N-1}$ is the measure of the surface of the unit ball of $\R^N$, plays a key role.
We stress that $\omega_{N-1}$ tends to zero as the dimension goes to $+ \infty$.

We begin to remark that if  $G$ satisfies $(G_0)$ with $\alpha < \alpha_N$ then 
\begin{equation}\label{rottt}
	m_1(G,N) := \sup_{u \in \cB_1} \Phi(u) <\infty.
\end{equation}
Indeed, let $u \in \cB_1 \cap L^\infty(B_1)$, and let $v:= 1_{B_1}G(u)$. Then $v \le c_1 e^{\alpha |u|^{N/(N-1)}}$ and 
$$
v \ln v \le c_1 e^{\alpha |u|^{N/(N-1)}} \Bigl(\alpha |u|^{N/(N-1)} + \ln c_1 \Bigr) \le c_2 e^{\alpha_N |u|^{N/(N-1)}}
$$
with a constant $c_2>0$. By the Trudinger-Moser inequality 
 \cites{trudinger,moser}
\begin{equation}\label{dis}
	\sup_{u \in \cB_1} \int_\Omega e^{\alpha_N |u|^{N/(N-1)}} dx = c(B_1)  < + \infty 
\end{equation}
it thus follows that 
\begin{equation}
	\label{eq:remark-intro-3}
	|v|_1 \le c_1 \int_{B_1}e^{\alpha_N |u|^{N/(N-1)}}dx \le c_1 c(B_1) =:c_3 \qquad \text{and}\qquad\int_{\R^N} v \ln v \,dx \le c_3 c(B_1)=: c_4. 
\end{equation}
From (\ref{rottt}) and the logarithmic Hardy-Littlewood-Sobolev inequality
\cite[Theorem 2]{beckner},
we infer that 
\begin{align*}
	\label{eq:remark-intro-4}
	\Phi(u)= \int_{\R^N} \int_{\R^N} \ln \frac{1}{|x-y|}v(x)v(y)\,dx dy 
	&\le \frac{|v|_{1}}{2}\Bigl(|v|_1 \bigl(c_5 +\bigl|\ln |v|_1\bigr|\bigr) + \int_{\R^N} v \ln v\,dx\Bigr)\\
	&\le \frac{c_3}{2}\bigl(c_3 \bigl(c_5 + |\ln c_3|\bigr) + c_4\bigr)<\infty, 
\end{align*}
where $c_5 >0$ is an explicit constant given in \cite[Theorem 2]{beckner}. For general $u \in \cB_1$, the same inequality follows by approximation, and hence we have $m_1(G,N) < \infty$. We emphasize that this argument does not apply in the case where $G$ has critical growth.

%To clarify this point, we split the kernel $\ln \frac{1}{|\,\cdot\,|}$ into its positive and negative part and define functionals 
%$\Phi_\pm: \cM(B_1) \to [0,\infty]$ by 
%\begin{align}
%  \label{eq:Phi-plus}
%\Phi_+(u)&=  \int_{B_1} \int_{B_1} \ln^+ \!\frac{1}{|x-y|}\, G(u(x))G(u(y))\,dx dy;\\ 
%\Phi_-(u)&=  \int_{B_1} \int_{B_1} \ln^+\! |x-y|\, G(u(x))G(u(y))\,dx dy, 
 % \label{eq:Phi-minus}
%\end{align}
%where $\ln^+ = \max \{\ln,0\}$. Here, for a measurable subset $\Omega \subset \R^N$, we let $\cM(\Omega)$ denotes the space of (Lebesgue-)measurable 
%functions $\Omega \to \R$. As we shall see in Section~\ref{sec:preliminaries} below, it follows from (\ref{eq:general-growth-condition}) that $\Phi_\pm(u)<\infty$ for %every $u \in W^{1,N}(B_1)$, and therefore the quantity in (\ref{eq:logarithmic-convolution-kernel}) has a well-defined value 
%\begin{equation}
 % \label{eq:def-Phi-new}
%\Phi(u):=\int_{B_1} \int_{B_1} \ln \!\frac{1}{|x-y|}\, G(u(x))G(u(y))\,dx dy = \Phi_+(u) -\Phi_-(u) \quad \in (-\infty,\infty)  
%\end{equation}
%for every $u \in W^{1,N}(B_1)$.

% In our first main result, we provide a sharp %borderline condition for the maximization %problem related to maximizing $\Phi$ within the % set $\cB_1$. 

As in 
\cite{CiWe2}, in order to derive  Trudinger-Moser type inequalities in $\cB_1$ and  the critical nonlinear growth rates for these inequalities,
we need to distinguish different forms of asymptotic growth of the nonlinearity $G$.

\smallskip

\begin{definition}
	Let $\beta \in \R$. We say that $F:\R \to \R$ has 
        \begin{itemize}
        \item[(i)] at most $\beta$-critical growth if $|F(s)| \leq c e^{\alpha_N |s|^{N/(N-1)}}(1+|s|)^\beta$ for $s \in \R$ with some constant $c>0$.
        \item[(ii)] at least $\beta$-critical growth if there exist $s_0, c>0$ with the property that 
$$
|F(s)| \ge c\ e^{\alpha_N |s|^{N/(N-1)}}|s|^\beta\qquad \text{for $|s| \ge s_0$.}
$$        
\end{itemize}
\end{definition}

\bigskip

In the following theorem, we deal with the critical growths.
\begin{theorem}
	\label{sec:introduction-main-thm-C-F2}
	Suppose that $G$ satisfies $(G_0)$.
	\begin{enumerate}
		\item If $G$ has at most $\beta$-critical growth for some $\beta \le -\frac{N}{2(N-1)}$, then 
$$
m_1(N, G) := \sup_{u \in \cB_1} \Phi(u) <\infty.
$$
		\item If $G$ has at most $\beta$-critical growth for some $\beta < -\frac{N}{2(N-1)}$, then $m_1(N,G)$ is attained, and every maximizer for $\Phi$ in $\cB_1$ is, up to sign, a radial and radially decreasing function in $\cB_1$.\\  
\item If $G$ has at least $\beta$ critical growth for some $\beta>-\frac{N}{2(N-1)} $, then 
$m_1(N,G)= \infty$.
\end{enumerate}
\end{theorem}

\medskip
We remark that it remains an open problem whether the functional $\Phi$ attains a maximum in $\cB_1$ if $G$ has at most $\beta$-critical growth for $\beta = -\frac{N}{2(N-1)}$. 

For $N=2$ and $\beta =-1$, a delicate blow-up analysis has been performed for special nonlinearity $g$ in the recent paper \cite{CiWeYu}. We conjecture that a blow up analysis can be performed in the quasilinear context to prove extremal functions in correspondence of the threshold $\beta = -\frac{N}{2(N-1)}$.
\smallskip
Next, we wish to state  logarithmic Trudinger-Moser type inequalities in the entire space $\R^N$. 

We define 
 $$
 \cB_{\infty} := \{u \in W^{1,N}(\R^N)\::\: \|u\|_N \le 1\}
 $$ 
where  $\|u\|^N_N := ( |\nabla u|^N + |u|^N )$ is the standard norm on $W^{1,N}(\R^N)$. 
 
 \smallskip
We study the maximization problem
\begin{equation}
	\label{eq:defwa}
	m_\infty(N,G):= \max_{u \in \cB_\infty} \Psi(u) 
\end{equation}
where 
\begin{equation}
	\label{eq:def-Psi-new}
	\Psi(u):= \int_{\R^N} \int_{\R^N} \ln \!\frac{1}{|x-y|}\, G(u(x))G(u(y))\,dx dy. 
\end{equation}
In view of the Moser-Trudinger inequality in $\R^N$ \cite{ruf, ruf2}, 
we assume additionally on $G$:  
\begin{itemize}
	\item[$(G_1)$] $G: \R \to [0,\infty)$ satisfies $(G_0)$, and $G(t)=O(|t|)$ as $t \to 0$.
\end{itemize}

We notice that  as  a consequence of assumption $(G_1)$, the functional 
$$\Psi_+(u)=  \int_{\R^N} \int_{\R^N} \ln^+ \!\frac{1}{|x-y|}\, G(u(x))G(u(y))\,dx dy 
<\infty$$ for every $u \in W^{1,N}(\R^N)$. 
Therefore set 
\begin{align*}
	%\label{eq:Psi-plus}
	\Psi_-(u)&=  \int_{\R^N} \int_{\R^N} \ln^+\! |x-y|\, G(u(x))G(u(y))\,dx dy 
\end{align*}
 we have 
 \begin{equation}
	\label{eq:def-Psi-new}
	\Psi(u):= \int_{\R^N} \int_{\R^N} \ln \!\frac{1}{|x-y|}\, G(u(x))G(u(y))\,dx dy =
	\Psi_+(u) -\Psi_-(u) \quad \in [-\infty,\infty)  
\end{equation}
is well-defined for every $u \in W^{1,N}(\R^N)$.

\smallskip
We state our main result.

\begin{theorem}
\label{sec:introduction-main-thm}
Suppose that $G$ satisfies $(G_1)$.
\begin{enumerate}
\item If $G$ has at most $\beta$-critical growth for some $\beta \le -\frac{N}{2(N-1)}$, then 
$$
m_\infty(N, G):= \sup_{u \in \cB_\infty}\Psi <\infty.
$$
\item If $G$ has at most $\beta$-critical growth for some $\beta < -\frac{N}{2(N-1)}$, then the value $m_\infty(N, G)$ is attained, and every maximizer for $\Psi$ in $\cB_\infty$ is, up to sign and translation, a radial and radially decreasing function in $\cB_\infty$.\\
\item If $G$ has at least $\beta$ critical growth for some $\beta>-\frac{N}{2(N-1)}$, then 
$$
m_\infty(N, G)= \infty
$$
in the sense that there exists a sequence $(u_n)_n$ in $\cB_\infty$ with $\Psi_\pm(u_n)< \infty$ for every $n \in \N$ and $\Psi(u_n) \to \infty$ as $n \to \infty$.
\end{enumerate}
\end{theorem}

\medskip
In Theorem \ref{sec:introduction-main-thm}, using  the Radial Lemma in $\R^N$, we pass from the maximization problem on $\R^N$ to the one on the ball. Compared to \cite{CiWe2}, this construction is more involved in dimension $N \geq 3$ and requires new delicate estimates.

Finally, as Theorem~\ref{sec:introduction-main-thm-C-F2}(i) and Theorem~\ref{sec:introduction-main-thm}(i) yield the existence of maximizers, the question arises whether these maximizers satisfy a corresponding Euler-Lagrange equation. Due to the weak growth conditions imposed on the nonlinearity $G$, this is not immediate. We can give an answer to this problem under natural additional assumptions. 

\begin{theorem}
\label{Euler-Lagrange-introduction}
Suppose that $G \in C^1(\R)$, and suppose that $g:= G'$ satisfies 
          \begin{equation}
            \label{derivative-growth-condition-introduction}
g(t) \le c e^{\alpha |t|^{N/(N-1)}}  \qquad \text{for $t \in \R$ with some constants $\alpha,c>0$.}
          \end{equation}
\begin{enumerate}
\item If $G$ satisfies $(G_0)$ and $u$ is  a maximizer of $\Phi$ on  $\cB_1$, then there exists $\theta \in \R$ such that $u$ satisfies the associated Euler-Lagrange equation in weak sense, i.e., 
	\begin{equation}
	\label{weak-e-l-B-1-intro}
	\int_{B_1}|\nabla u|^{N-2} \nabla u \nabla \phi \,dx = \theta 
	\int_{B_1} \ln \frac{1}{|\cdot|}*\bigl(1_{B_1} G(u)\bigr)g(u)\phi\,dx \qquad \text{for all $\phi \in W^{1,N}_0(B_1).$}
	\end{equation}
 
\item If $G$ satisfies $(G_1)$ and $u$ is a maximizer of $\Psi$ on $\cB_\infty$, then there exists $\theta \in \R$ such that $u$ satisfies the associated Euler-Lagrange equation weak sense, i.e., 
	\begin{equation}
	\label{weak-e-l-R-2-intro}
	\int_{\R^N}\bigl(|\nabla u|^{N-2}\nabla u \nabla \phi + |u|^{N-2} u \phi\bigr) \,dx = \theta 
	\int_{\R^N} (\ln \frac{1}{|\cdot|}*G(u))g(u)\phi\,dx 
	\end{equation}
for all $\phi \in W^{1,N}_0(\R^N)$ with bounded support. 
\end{enumerate}
\end{theorem}

\bigskip
Differently to the planar case, 
the maximation problem in $\R^N$ leads to 
a quasilinear Choquard equation.
As noticed in \cite{BucurCassani}, (see also Remark (\ref{syst}))
the integro-differential problem 
$(\ref{weak-e-l-R-2-intro})$ in $\R^N$ can be derived by the $N$-Laplacian Schr\"odinger equation strongly coupled with a higher order fractional Poisson equation. When the order of the Riesz potential $\alpha$ is equal to the Euclidean dimension $N$,  the system turns out to be related to a nonlocal logarithmic Choquard type equation.
Let  $I_N(x) = \frac{1}{\gamma_N}\log \frac{1}{|x|}$ be the logarithmic Riesz potential. By setting $w:= I_N \ast G(u)$, the equation
(\ref{weak-e-l-R-2-intro}) becomes formally equivalent
to 
the nonlinear system
\begin{equation}
	\label{eq:171-intro}
	\left\{
	\begin{aligned}
		-\Delta_N u + |u|^{N-2}u &= \theta \,  w g(u)    &&\qquad \text{in $\R^N$,}\\
		(-\Delta)^{N/2} w &=   G(u)  &&\qquad \text{in $\R^N$}.\\
	\end{aligned}
	\right.
\end{equation}
More precisely, if $u$ is a solution of the logarithmic Choquard equation, then the couple $(u, I_N \ast G(u))$ gives a distributional solution of the Schr\"odinger-Poisson system. The contrary it is not still clear, in particular with respect to the functional setting (see \cite{BucurCassani} and also
\cite{CassaniLiuRomani,CassaniLiuRomani1} for the quasilinear fractional counterpart).

It remains a challenging question to investigate the qualitative properties of the solution pair $(u, w)$  of the quasilinear and fractional system (\ref{eq:171-intro}) by moving plane techniques.  For the planar case some qualitative results has been obtained in \cite{CiWe2}. We also mention the recent paper \cite{CassaniLiuRomani} for the fractional quasilinear setting under additional regularity condition on the solution.

An other open problem is the derivation of 
 Trudinger-Moser inequalities for arbitrary subdomains of $\R^N$. 
 As conjectured in \cite{CiWe2} for $N=2$,
 the question whether maximizers exist in critical cases seems quite challenging and might depend on the geometry of the domain.

\medskip

The paper is organized as follows. In Section 2, we establish some preliminaries related to some nonlocal interaction energies. In Section 3 we study  Trudinger-Moser inequalities with logarithmic convolution potentials when the domain is a ball and we establish Theorem \ref{sec:introduction-main-thm-C-F2}.
Section 4 is concerned with the maximization problem when the domain is  the entire space  $\R^N$ and contains the proof of Theorem \ref{sec:introduction-main-thm}. 
In Section 5 we show that such extremal functions satisfy Euler Lagrange equations in a weak sense, and we prove Theorem \ref{Euler-Lagrange-introduction}.

\section{Two quadratics forms}
\label{sec:preliminaries}

\bigskip

We first introduce some notation. We denote by $\cM(\R^N)$  the space of real-valued measurable functions $\R^N$, and  we let $\cM_+(\R^N)$ to denote the subset of nonnegative functions in $\cM(\R^N)$. If $\Omega \subset \R^N$ is a measurable subset and $u$ is a real-valued measurable function on $\Omega$, we also regard $u$ as a function in $\cM(\R^N)$ by trivial extension. We then define the quadratic forms $b_\pm: \cM_+(\R^N) \to [0,\infty]$ by 
\begin{equation}\label{2.1eq}
    (v,w) \mapsto b_{+}(v,w)=\int_{\mathbb{R}^N} \int_{\mathbb{R}^N} \ln^{+} \frac{1}{\abs{x-y}} v(x)v(y) dxdy,  
\end{equation}

\begin{equation}\label{2.2eq}
    (v,w) \mapsto b_{-}(v,w)=\int_{\mathbb{R}^N} \int_{\mathbb{R}^N} \ln^{+} \abs{x-y} v(x)v(y) dxdy.  
\end{equation}
Moreover, we define \begin{equation*}
    b_{0}(v,w):=b_{+}(v,w)-b_{-}(v,w) \in [-\infty, \infty)
\end{equation*}
for all functions $v$,$w \in \mathcal{M}(\mathbb{R}^N)$ for which $b_{+}(v,w)< \infty$. For the sake of brevity, we also set \begin{equation*}
    b_{\pm}(v)=b_{\pm}(v,v) \text{\ \ \ \ and\ \ \ \ }  b_0(v):=b_0(v,v) \text{\ \ if \ \ } b_{+}(v)<\infty.
\end{equation*} Next we define
\begin{equation*}
    \abs{v}_{*}:=\int_{\mathbb{R}^N} \ln(1+\abs{x}) \abs{v}dx \text{\ \ \ } \in [0, \infty], \text{\ \ \ for\  } v \in \mathcal{M}(\mathbb{R}^N)  
\end{equation*} and we consider the space \begin{equation*}
    L^{1}_{ln}(\mathbb{R}^N):=\{ v \in L^{1}(\mathbb{R}^N), \text{\ \ } \abs{v}_{*} < \infty \}.
\end{equation*}
Taking into account that %\begin{equation}\label{2.3eq}
 %   \ln^{+}\abs{x-y}\leq \ln^{+}(\abs{x}+ %\abs{y})\leq (1+\abs{x})(1+ %\abs{y})=\ln(1+\abs{x})+ \ln(1+\abs{y}) %\text{for} x,y \in \mathbb{R}^N,
%\end{equation}
%we infer that 
\begin{equation*}
    [\ln^{+} \abs{\cdot} \ast v](x)\leq \ln(1 + \abs{x})\abs{v}_1+\abs{v}_{*}\,\,\,\,\,\,\,\,\,\,\,\, x \in \mathbb{R}^N
\end{equation*}
we have 
\begin{equation}\label{2.4eq}
\ln^{+} \abs{\cdot} \ast v  \in L^{\infty}_{loc}(\mathbb{R}^N)\text{\ \ \ \ for\ \ } v \in L^{1}_{ln}(\mathbb{R}^N)
    \end{equation}
We also infer that 
\begin{equation}\label{2.5eq}
    b_{-}(v,w)\leq \abs{v}_1 \abs{w}_* + \abs{w}_1\abs{v}_*< +\infty \text{ $ $ $ $for } v,w \in L_{ln}^1(\mathbb{R}^N)\cap \mathcal{M}_+(\mathbb{R}^N).
\end{equation}
\begin{lemma}
\label{converse-inequality}
Let $v,w \in \cM_+(\R^N)$ be given with $v \not \equiv 0$, $w \in L^1_{loc}(\R^N)$ and $b_-(v,w)<\infty$. Then $w \in L_{ln}^1(\R^N)$.   \end{lemma}

\begin{proof}
Since $v \not \equiv 0$, there exist constants $R>1,c>0$ and a measurable subset $A \subset B_{R}$ of positive measure with $v \ge c$ on $A$. Then we have $|x-y| \ge \frac{|y|}{2}\ge R$ for $x \in A$, $y \in \R^{N} \setminus B_{2R}$ and therefore
\begin{equation*}
    b_-(v,w)=\int_{\mathbb{R}^N}\int_{\mathbb{R}^N} \ln ^+\abs{x-y}v(x)w(y)dx dy\geq c\int_{A}\int_{\mathbb{R}^N} w(y)\ln ^+\abs{x-y} dy
\end{equation*}
\begin{equation*}
    \geq c\abs{A} \int_{\mathbb{R}^N\setminus B_{2R}}  w(y)\ln ^+\frac{\abs{y}}{2}dy\geq C_1 \int_{\mathbb{R}^N\setminus B_{2R}} w(y) \ln^{+}(1+\abs{y})dy
\end{equation*}
\begin{equation*}
    \geq c_1 \bigg(\int_{\mathbb{R}^N} w(y)\ln (1+\abs{y}) dy-\ln (1+2R)\int_{B_{2R}}w(y)dy\bigg)
\end{equation*}
\begin{equation*}
    = c_1 (\abs{w}_*-\ln (1+2R)\norm{w}_{L^1(B_{2R})})
\end{equation*}
with a constant $c_1>0$. Hence $|w|_*< \infty$, and therefore $w \in L_{ln}^1(\R^N)$.
\end{proof}

The following corollary is immediate.

\begin{corollary}
\label{converse-inequality-remark} 
If $v \in L^1_{loc}(\R^N)$ is nonnegative and satisfies $b_-(v)<\infty$, then $v \in L_{ln}^1(\R^N)$.
\end{corollary}

We also note that by (\ref{2.5eq}) we have 
\begin{equation}
  \label{eq:b-second-ineq}
b_-(v,w) \le |v|_1 |w|_* +|w|_1 |v|_* \le \bigl(2 \ln 2\bigr) |v|_1 |w|_1< \infty 
\end{equation}
for $v,w \in L^1(\R^N)$ with $v \equiv w \equiv 0$ on $\R^N \setminus B_1$.\\

\medskip

Next, let $v^*$ denote the Schwarz symmetrization of a function $v \in \cM(\R^N)$. Thus $v^* \in \cM_+(\R^N)$ is radial and nonincreasing in the radial variable. We recall the following Riesz rearrangement type inequalities that can be derived arguing as in \cite{CiWe}.

\begin{lemma}
\label{Riesz-rearrangement}
Let $v \in \cM_+(\R^N)$. Then we have:
\begin{enumerate}
\item[(i)] $b_+(v^*) \geq b_+(v)$;
\item[(ii)] $b_-(v^*) \leq b_-(v)$;
\item[(iii)] If $b_+(v^*) < \infty$, then 
  \begin{equation}
\label{eq:Riesz-rearrangement-B-0}
b_0(v^*)\ge b_0(v).
\end{equation}
If, in addition, $v \in L^p(\R^N)$ for some $p \in [1,2]$ and  $b_+(v)< \infty$ and $b_-(v)< \infty$, then equality holds in (\ref{eq:Riesz-rearrangement-B-0}) if and only if $v = v^*(\cdot - x_0)$  for some $x_0 \in \R^N$.
\end{enumerate}
 \end{lemma}

Here we note that, under  assumption (iii), both $b_0(v^*)$ and $b_0(v)$ are well defined by the inequalities in (i) and (ii).

\medskip

\begin{lemma}
	\label{sec:newtons-theorem}
	Let $v \in L_{ln}^1 (\R^N)$ be a nonnegative radial function. Then 
	the convolution $[(\ln |\cdot|) * v](x) \in \R$ is well defined for $x \in \R^N \setminus \{0\}$ and given by 
	\begin{equation}
		\label{eq:pointwise-convolution}
		[(\ln |\cdot|) * v](x)= \ln |x| \int_{B_{|x|}} v(y)dy + \int_{\R^N \setminus B_{|x|}} \ln |y| v(y)dy 
	\end{equation}
\end{lemma}

\begin{proof}
	Let $x \in \R^N \setminus \{0\}$. For nonnegative radial functions $v \in L^\infty(\R^N)$ with bounded support, both the LHS and RHS of (\ref{eq:pointwise-convolution}) are well-defined, and the formula holds by Newton's theorem (see for instance \cite[Theorem 9.7]{LiebLoss}).
	
	Next we let $v \in L_{ln}^1 (\R^N)$ be a nonnegative radial function. We have 
	\begin{equation}
	[(\ln^+ \! |\cdot|  * v](x) \le \ln (1+ |x|) |v|_1 + |v|_* < \infty.
	\end{equation}

	We consider an increasing sequence of nonnegative radial functions $v_n \in L^\infty(\R^N)$ with bounded support and such that $v_n \to v$ pointwisely. Then by monotone convergence theorem
	$$
	[(\ln^+ \!|\cdot|) * v_n](x) \le [(\ln^+ \! |\cdot|  * v](x) \quad \text{for all $n$} \qquad \text{and}\qquad  \lim_{n \to \infty} [(\ln^+ \!|\cdot|) * v_n](x) = [(\ln^+ \! |\cdot|  * v](x).
	$$ Moreover, again by monotone convergence, 
	\begin{align*}
		[(\ln^+ \!\frac{1}{|\cdot|}) * v](x)& = \lim_{n \to \infty} 
		[(\ln^+ \!\frac{1}{|\cdot|}) * v_n](x)= \lim_{n \to \infty} \Bigl(
		[(\ln^+ \!|\cdot|) * v_n](x)- 
		[(\ln |\cdot|) * v_n](x)\Bigr)\\ 
		&= [(\ln^+ \! |\cdot|  * v](x)- \Bigl( \ln |x| \int_{B_{|x|}} v(y)dy + \int_{\R^N \setminus B_{|x|}} \ln |y| v(y)dy \Bigr). 
	\end{align*}
	We thus conclude $(\ref{eq:pointwise-convolution})$.
\end{proof}

From Lemma~\ref{converse-inequality}, Lemma~\ref{sec:newtons-theorem} and Fubini's theorem, we have the following result.

\begin{corollary}
	\label{cor-newtons-theorem}
	Let $v,w \in L^1_{loc}(\R^N)$ be nonnegative radial functions. 
	\begin{enumerate}
		\item If $b_\pm(v,w)< \infty$ and $v \not \equiv 0$, $w \not \equiv 0$, then $v,w \in L_{ln}^1(\R^N)$ and 
		\begin{align}
			\frac{b_0(v,w)}{\omega_{N-1}^2}&=\int_{0}^{\infty} r^{N-1} v(r)\bigg( \ln \frac{1}{r} \int_{0}^{r}\rho^{N-1}w(\rho)d\rho +\int_{r}^{+\infty}\rho^{N-1}\ln\dfrac{1}{\rho}w(\rho)d\rho\bigg)dr
			\\  
			& = \int_{0}^{1} r^{N-1} w(r)\bigg( \ln \frac{1}{r} \int_{0}^{r}\rho^{N-1}v(\rho)d\rho +\int_{r}^{+\infty}\rho^{N-1}\ln\dfrac{1}{\rho}v(\rho)d\rho\bigg)dr.
		\end{align}
		\item If $b_\pm(v)< \infty$, then $v \in L_{ln}^1(\R^N)$ and 
		\begin{equation*}
			\dfrac{b_0(v)}{\omega_{N-1}^2}=2\int_{0}^{\infty} r^{N-1} v(r)\ln \frac{1}{r} \int_{0}^{r}\rho^{N-1}v(\rho)d\rho dr. 
		\end{equation*}
	\end{enumerate}
\end{corollary}

\bigskip

\section{Preliminary estimates: the unit ball versus the entire space}

From \cite[Lemma 7.17]{LiebLoss}, we can deduce the following P\'olya-Szego  inequality.

  \begin{lemma}
	\label{symmetrized-invariance-H-1}   
	For any $u \in W^{1,N}(\R^N)$ we have $u^* \in W^{1,N}(\R^N)$ and 
	$$
	|u^*|_N = |u|_N, \quad |\nabla u^*|_N \le |\nabla u|_N 
	$$
	Consequently, we have $u^* \in \cB_1$ if $u \in \cB_1$ and $u^* \in \cB_\infty$ if $u \in \cB_\infty$. 
\end{lemma}

\smallskip

 \begin{lemma}
 \label{lemma-nonuniform-finiteness-convolution}
Let $g: \R \to \R$ be a continuous function satisfying 
\begin{equation}
  \label{eq:general-growth-condition-lemma}
|g(t)| \le c e^{\alpha |t|^{N/(N-1)}} \qquad \text{for $t \in \R$ with constants $c,\alpha>0$.}
\end{equation}
\begin{enumerate}
\item[(i)] If $u \in W^{1,N}_0(B_1)$, then we have $g(u) \in L^s(B_1)$ for $1 \le s < \infty$ and
$$ 
\ln^+ \frac{1}{|\cdot|} * \bigl(1_{B_1} g(u)\bigr), \quad  \ln^+ |\cdot| * \bigl(1_{B_1} g(u)\bigr) \quad \in L^\infty(B_1).
$$
\item[(ii)] If $u \in W^{1,N}(\R^N)$, then $g(u) \in L^s_{loc}(\R^N)$ for $1 \le s < \infty$.\vspace{0.1cm} 
 
If, in addition, $|g(t)|=O(|t|)$ as $t \to 0$, then 
$g(u) \in L^2(\R^N)$ and 
$$
\ln^+ \frac{1}{|\cdot|} * g(u) \quad \in L^s(\R^N) \qquad \text{for $2 \le s \le \infty$.} 
$$
\end{enumerate}
\end{lemma}
\begin{proof}
We first recall that 
\begin{equation}
\label{general-nonuniform-growth-estimate}
\int_{B_1}e^{\delta |u|^{N/(N-1)}}dx < \infty \qquad \text{for any $u \in W^{1,N}_0(B_1)$ and any $\delta>0$}
\end{equation}
as shown in \cite{trudinger}  (see also \cite[p.195]{Lions}).
To prove (i), we now let $u \in W^{1,N}_0(B_1)$ and define $v= 1_{B_1} g(u)$. By (\ref{eq:general-growth-condition-lemma}) and (\ref{general-nonuniform-growth-estimate}), we have $v \in L^s(\R^N)$ for $s \in [1,\infty)$. In fact 
\begin{equation*}
    \int_{\mathbb{R}^N}\abs{v}^s dx=\int_{B_{1}}\abs{g(u)}^s dx \leq c^s \int_{B_{1}}(e^{\alpha |u|^{N/(N-1)}})^s dx \,\textless +\infty.
\end{equation*}
Moreover,
\begin{equation*}
\int_{\mathbb{R}^N}\bigg|\ln^+\dfrac{1}{\abs{x}}\bigg|^s dx=\int_{\mathbb{R}^N}\bigg(\ln^+\dfrac{1}{\abs{x}}\bigg)^s dx=\omega_{N-1}\int_0^1 \ln^s\dfrac{1}{r}r^{N-1}dr\,\textless +\infty;
\end{equation*}
\begin{equation}
  \label{eq:ln-plus-L-p}
\ln^+ \frac{1}{|\cdot|} \in L^s(\R^N) \qquad \text{for $s \in [1,\infty)$.}  
\end{equation}
Choosing $s=2$, we deduce by Young's inequality that $\ln^+\frac{1}{|\cdot|} * v \in L^\infty(\R^N)$. Moreover, we deduce that 
$$
[\ln^+|\cdot| * v](x) = \int_{B_1}\ln^+|x-y|v(y)\,dy < \ln 2 |v|_1 \qquad \text{for $x \in B_1$}
$$
and therefore $\ln^+|\cdot| * v \in L^\infty(B_1)$.\\ 
(ii) Let $u \in W^{1,N}(\R^N)$. For given $R>0$ and $s \in [1,\infty)$, we wish to prove that $1_{B_R}g(u) \in L^s(\R^N)$. For this, by (\ref{eq:general-growth-condition-lemma}), we may  assume without loss of generality that $g(t)=e^{\alpha |t|^{N/(N-1)}}$ for $t \in \R$. Then we have $[1_{B_R}g(u)]^* \le 1_{B_R}[g(u)]^* = 1_{B_R}g(u^*)$ and therefore 
$$
\bigl|1_{B_R}g(u)\bigr|_s = \bigl|[1_{B_R}g(u)]^*\bigr|_s \le \bigl| 1_{B_R}g(u^*) \bigr|_s,
$$
so it suffices to consider the case where $u=u^*$ in the following. Let $v := g(u)$. Since $u$ is locally bounded on $\R^N \setminus \{0\}$, the same is true for 
$v$. It thus suffices to prove that $v \in L^s(B_1)$. For this we consider the function
\begin{equation*}
   U := \Bigl(1 +u(1)^{\frac{N}{N-1}}\Bigr)^{\frac{N-1}{N}}[u^{\frac{N}{2(N-1)}}-u^{\frac{N}{2(N-1)}}(1)]_+^{\frac{2(N-1)}{N}}.
\end{equation*}
We prove that $U\in W^{1,N}_0(B_1)$. Indeed \begin{align*}
    &\int_{B_1} |\nabla U|^N\, dx =  \Bigl(1 +u(1)^{\frac{N}{N-1}}\Bigr)^{N-1}\int_{B_1} |\nabla [u^{\frac{N}{2(N-1)}}-u^{\frac{N}{2(N-1)}}(1)]_+^{\frac{2(N-1)}{N}}|^N\, dx  \\ &
    \leq 2^{N-2} \Bigl(1+ u(1)^N\Bigr) \int_{B_1}   [u^{\frac{N}{2(N-1)}}-u^{\frac{N}{2(N-1)}}(1)]_+^{N-2}  u^{\frac{(2-N)N}{2(N-1)}} |\nabla u|^N\, dx  \\ & \leq 2^{N-2} \Bigl(1+u(1)^N\Bigr) \int_{B_1} u^{\frac{N(N-2)}{2(N-1)}} u^{\frac{(2-N)N}{2(N-1)}} |\nabla u|^N\, dx= 2^{N-2} \Bigl(1+ u(1)^N\Bigr) \int_{B_1}|\nabla u|^N\, dx< + \infty.
\end{align*}

Let \begin{equation*}
    \phi=u^{\frac{N}{2(N-1)}}
\end{equation*}for which holds 
\begin{equation*}
   \phi^2=[\phi-\phi(1)]^2+2\phi(1)[\phi-\phi(1)]^+ +\phi^2(1)\leq [\phi-\phi(1)]^2+\phi^2(1)[\phi-\phi(1)]^2+1 +\phi^2(1)
\end{equation*}

\begin{equation} \label{test}
    =\phi^2(1)+1+[\phi-\phi(1)]^2(1+\phi^2(1)).
\end{equation}
By \eqref{test}, we have 
 \begin{equation*}
u^{\frac{N}{N-1}} \leq u^{\frac{N}{N-1}}(1)+1+ \Bigl(1+ u^{\frac{N}{N-1}}(1)\Bigr)\Bigl[u^{\frac{N}{2(N-1)}}-u^{\frac{N}{2(N-1)}}(1)\Bigr]^{2}_+
\end{equation*}
and thus, by the definition of $U$,
\begin{equation}\label{contocostantepag8}
u^{\frac{N}{N-1}}\leq u^{\frac{N}{N-1}}(1) +1+U^{\frac{N}{N-1}}.
\end{equation}

From \eqref{contocostantepag8} we have
$$
\int_{B_1} |v|^s\,dx = \int_{B_1} e^{ s\alpha u^{N/(N-1)}}\,dx \le 
e^{s \alpha \bigl(1+ u^{N/(N-1)}(1)\bigr)} \int_{B_1}e^{s \alpha U^{N/(N-1)}}\,dx < \infty
$$
by (\ref{general-nonuniform-growth-estimate}). It thus follows that $v \in L^s_{loc}(\R^N)$.\\
Next we assume that $|g(t)|=O(|t|)$ as $t \to 0$. To prove that $v=g(u) \in L^2(\R^N)$, by  (\ref{eq:general-growth-condition-lemma}) it suffices to consider the case where $g(t)=e^{\alpha t^{N/(N-1)}}-1$. 
It then follows that $|g(u)|_N = |[g(u)]^*|_N =|g(u^*)|_N$, so we may assume again that $u=u^*$.
Then $u$ is bounded on $\R^N \setminus B_1$, and therefore it follows 
$|v| \le C |u|$ on $\R^N \setminus B_1$, so $v 1_{\R^N \setminus B_1} \in L^2(\R^N)$ since $u \in L^2(\R^N)$. Since we already proved that 
$v 1_{B_1} \in L^2(\R^N)$, we infer that $v \in L^2(\R^N)$. By (\ref{eq:ln-plus-L-p}) and Young's inequality, it now follows that 
$\ln^+\frac{1}{|\cdot|} * v \in L^s(\R^N)$ for $2 \le s \le \infty$.
\end{proof}

From  Lemma~\ref{lemma-nonuniform-finiteness-convolution}, applied to $g=G$, we infer the following corollary.

\begin{corollary}$ $
\label{corollary-nonuniform-finiteness}
\begin{enumerate}
\item[(i)] If $G$ satisfies assumption $(G_0)$ and $u \in W^{1,N}_0(B_1)$, then $\Phi_\pm(u)<\infty$.
\item[(ii)] If $G$ satisfies assumption $(G_1)$ and $u \in W^{1,N}(\R^N)$, then $\Psi_+(u)<\infty$. 
\end{enumerate}
\end{corollary}

\begin{remark}
\label{m-F-greater-F-0}
(i) If $G$ satisfies $(G_0)$, it follows from Corollary~\ref{corollary-nonuniform-finiteness}(i) and Corollary~\ref{cor-newtons-theorem}(ii) that 
\begin{align*}
&\Phi(u)= 2\omega_{N-1}^2\int_{0}^{1} r^{N-1} G(u(r))\ln \frac{1}{r} \int_{0}^{r}\rho^{N-1}G(u(\rho))d\rho dr \\ & > 2\omega_{N-1}^2\int_{0}^{1} r^{N-1} G(u(0))\ln \frac{1}{r} \int_{0}^{r}\rho^{N-1}G(u(0))d\rho dr=\Phi(0)  
\end{align*}
for every $u \in W^{1,N}_0(B_1) \setminus \{0\}$. Hence $m_1(N, G) \in \bigl(\Phi(0), \infty \bigr]$,
and therefore $m_1(N, G)$ is not attained at $u=0$.\\
(ii) If $G$ satisfies $(G_1)$, we have 
$$
\Psi(u)=\Phi(u)>\Phi(0)= \Psi(0)=0\quad \text{for every function 
$u \in W^{1,N}_0(B_1) \subset W^{1,N}(\R^N)$ with $u \not \equiv 0.$}
$$ 
In particular, we have $m_\infty(N, G) \in \bigl(0, \infty \bigr]$, 
and therefore $m_\infty(N, G)$ is not attained at $u=0$.
\end{remark}
\begin{lemma}
\label{general-boundedness}
Let $\beta_1, \beta_2\le 0$ satisfy $\beta_1 + \beta_2 \le -\dfrac{N}{N-1}$, and let $\cB_{1,rad}:= \{u \in \cB_1\::\: \text{$u$ radial}\}$. 
Then
\begin{equation}
  \label{eq:def-phi-beta-1-beta-2}
\Phi_{\beta_1,\beta_2}(u_1,u_2):= \omega_{N-1}^2\int_0^1 r^{N-1}(1+\abs{u_1(r)})^{\beta_1}e^{\alpha_N |u_1|^{N/(N-1)}(r)}\ln\dfrac{1}{r} \times\end{equation} \begin{equation*}
\times \int_0^r \rho^{N-1}(1+\abs{u_2(\rho)})^{\beta_2}e^{\alpha_N|u_2|^{N/(N-1)}(\rho)}d\rho\, dr
\end{equation*}
defines a bounded functional $\Phi_{\beta_1,\beta_2}: \cB_{1,rad} \times \cB_{1,rad} \to [0,\infty)$. 
\end{lemma}
\begin{proof}
We fix $0 < \epsilon < 1 /\alpha_N$. Moreover, we let $u_1,u_2 \in \cB_{1,rad}$, and we define 
$$
A_i^+:= \{r \in (0,1]\::\: u_i(r) \ge  [ \eps(-\ln r)]^{(N-1)/N}\},\qquad
A_i^-:= \{r \in (0,1]\::\: u_i(r) < [\eps(-\ln r)]^{(N-1)/N} \}
$$
for $i= 1,2$. With $v_i:= (1+|u_i|)^{\beta_i} e^{\alpha_N u_i^{N/(N-1)}}$ we then have
\begin{equation}
  \label{eq:A--estimate}
v_i(r) \le e^{\alpha_N |u_i|^{N/(N-1)}(r)} \le r^{-\alpha_N \eps}\qquad \text{for $r \in A_i^-$,}
\end{equation}
and
\begin{equation}
  \label{eq:A-+estimate}
v_i(r) \le \bigl(1+ [ \eps(-\ln r)]^{(N-1)/N}\bigr)^{\beta_i} e^{\alpha_N |u_i|^{N/(N-1)}(r)} \qquad \text{for $r \in A_i^+$.}
\end{equation}
Moreover, 
\begin{align*}
&\Phi_{\beta_1,\beta_2}(u_1,u_2)
= \omega_{N-1}^2\bigg(\int_{A_1^+} r^{N-1} v_1(r) \ln \frac{1}{r} \int_{A_2^+ \cap [0,r]} \rho^{N-1}  v_2(\rho)d\rho dr+\\ 
&+ \int_{A_1^-} r^{N-1} v_1(r) \ln \frac{1}{r} \int_{A_2^- \cap [0,r]} \rho^{N-1}  v_2(\rho)d\rho dr 
+ \int_{A_1^+} r^{N-1} v_1(r) \ln \frac{1}{r} \int_{A_2^- \cap [0,r]} \rho^{N-1}  v_2(\rho)d\rho dr+\\  
&+ \int_{A_1^-} r^{N-1} v_1(r) \ln \frac{1}{r} \int_{A_2^+ \cap [0,r]} \rho^{N-1}  v_2(\rho)d\rho dr\bigg),
\end{align*}
where 
\begin{align*}
\omega_{N-1}^2\int_{A_1^-} r^{N-1}v_1(r) \ln \frac{1}{r} \int_{A_2^- \cap [0,r]} \rho^{N-1}  v_2(\rho)d\rho dr&\le\omega_{N-1}^2  \int_{A_1^-} r^{N-1-\alpha_N \eps } \ln \frac{1}{r} 
\int_{A_2^- \cap [0,r]} \rho^{N-1-\alpha_N \eps }d\rho dr\\
&\le \omega_{N-1}^2\frac{1}{N-\alpha_N\eps} \int_{0}^1 r^{2N-1-2\alpha_N \eps } \ln \frac{1}{r} dr =C_1 < \infty
\end{align*}
and 
\begin{align*}
&\omega_{N-1}^2\int_{A_1^+} r^{N-1}v_1(r) \ln \frac{1}{r} \int_{A_2^+ \cap [0,r]} \rho^{N-1}  v_2(\rho)d\rho dr \\
&\le \omega_{N-1}^2\int_{0}^1 \bigl(1+ [\eps(-\ln r)]^{(N-1)/N}\bigr)^{\beta_1} e^{\alpha_{N} |u_1|^{(N-1)/N}(r)}\ln \frac{1}{r} \times \\ &\times   
\int_{A_2^+ \cap [0,r]} \rho^{N-1} 
\bigl(1+ [\eps(-\ln \rho)]^{(N-1)/N}\bigr)^{\beta_2} e^{\alpha_{N} |u_2|^{N/(N-1)}(\rho)} d\rho dr\\
&\le \omega_{N-1}^2 \int_0^1 (-\ln r)(1+ [\eps(-\ln r)]^{(N-1)/N}\bigr)^{\beta_1+\beta_2} 
r^{N-1} e^{\alpha_N |u_1|^{N/(N-1)}(r)} \int_0^r  
\rho^{N-1} e^{\alpha_N |u_2|^{N/(N-1)}(\rho)}d\rho dr\\
&\le \omega_{N-1}^2\frac{1}{\eps} \Bigl(\int_0^1 r^{N-1} e^{\alpha_N |u_1|^{N/(N-1)}(r)}dr \Bigr) \omega_{N-1} \Bigl(\int_0^1 \rho^{N-1} e^{\alpha_N |u_2|^{N/(N-1)}(\rho)}d\rho \Bigr)\le C_2.
\end{align*}
Here we used the assumption $\beta_1 + \beta_2 \le -\frac{N}{(N-1)}$ and the Trudinger-Moser inequality (\ref{dis}). Moreover, $C_1,C_2,\dots$ are positive constants independent of $u_1,u_2 \in \cB_{1,rad}$. We also have
\begin{align*}
&\omega_{N-1}^2\int_{A_1^+} r^{N-1} v_1(r) \ln \frac{1}{r} \int_{A_2^- \cap [0,r]} \rho^{N-1}  v_2(\rho)d\rho dr\\ &\le \omega_{N-1}^2  \int_{0}^1 r^{N-1} \bigl(1+ [\eps(-\ln r)]^{(N-1)/N}\bigr)^{\beta_1} e^{\alpha_N |u_1|^{N/(N-1)}(r)} \ln \frac{1}{r} \int_{0}^r 
\rho^{N-1-\alpha_N \eps } d\rho dr\\
&\le \omega_{N-1}^2 \frac{1}{N-\alpha_N \eps} \int_{0}^1 r^{2N-1-\alpha_N \eps} \Bigl(\ln \frac{1}{r}\Bigr)e^{\alpha_N |u_1|^{N/(N-1)}(r)}\,dr \le C_3 \omega_{N-1} \int_{0}^1 r^{N-1} e^{\alpha_N |u_1|^{N/(N-1)}(r)}\,dr \le C_4
\end{align*}
and
\begin{align*}
&\omega_{N-1}^2\int_{A_1^-} r^{N-1} v_1(r) \ln \frac{1}{r} \int_{A_2^+ \cap [0,r]} \rho^{N-1}  v_2(\rho)d\rho dr\\ & \le \omega_{N-1}^2 \int_{0}^1 r^{N-1-\alpha_N \eps } \ln \frac{1}{r}   \int_{0}^r 
\rho^{N-1} e^{\alpha_N |u_2|^{N/(N-1)}(\rho)} d\rho dr\\
& \le \omega_{N-1}  \int_{0}^1 r^{N-1-\alpha_N \eps} \ln \frac{1}{r}dr \omega_{N-1}\int_{0}^1 
\rho^{N-1} e^{\alpha_N |u_2|^{N/(N-1)}(\rho)} d\rho \le C_5 \omega_{N-1} \int_{0}^1 r^{N-1-\alpha_N \eps} \ln \frac{1}{r}\,dr \le C_6
\end{align*}
again by the Trudinger-Moser inequality (\ref{dis}).
The proof is thus finished. 
\end{proof}
$\bigskip$

\section{Maximization problem on the ball}\label{sec:finiteness-result}

In this section we complete the proof of Theorem~\ref{sec:introduction-main-thm-C-F2}. We recall that 
$\cB_1:= \{u \in W^{1,N}_0(B_1)\::\: |\nabla u|_N \le 1\}$, and  
$$
\cB_1^*:=  \{u^* \::\: u \in \cB_1\}
$$
denotes the corresponding Schwarz symmetrized set. By Lemma~\ref{symmetrized-invariance-H-1}, we then have $\cB_1^* \subset \cB_1$.
\begin{lemma}
\label{C-B-1-restimate}
Suppose that $G$ satisfies $(G_0)$ and has at most $0$-critical growth. Then the functional $\Phi_-$ is uniformly bounded on $\cB_1$.   
\end{lemma}

\begin{proof}
Let $u \in \cB_1$. By definition and (\ref{2.5eq}) we have 
$$
\Phi_-(u)= b_-(1_{B_1}G(u),1_{B_1}G(u)) \le \bigl(2 \ln 2\bigr) |1_{B_1}G(u)|_1^2
$$
where 
$$
|1_{B_1} G(u)|_1 \le c_1 \int_{B_1}e^{\alpha_N |u|^{N/(N-1)}}\,dx \le c_2
$$
with constants $c_1,c_2>0$ independent of $u$ by assumption and the Trudinger-Moser inequality (\ref{dis}).
\end{proof}

\begin{proposition}
\label{C-finiteness}
Suppose that $G$ satisfies $(G_0)$ and has at most $\beta$-critical growth for some $\beta\le -\frac{N}{2(N-1)}$. Then we have 
$$
m_1(N, G) \le m_1^+(N,G)< \infty,\qquad \text{where}\qquad m_1(N, G):= \sup_{\cB_1}\Phi \quad \text{and}\quad m_1^+(N, G):= \sup_{\cB_1}\Phi^+.
$$
\end{proposition}

\begin{proof}
By Lemma~\ref{Riesz-rearrangement}, we have 
$$
\Phi^+(u)= b_+\bigl(1_{B_1}G(u),1_{B_1}G(u)\bigr)\le b_+\bigl([1_{B_1}G(u)]^*,[1_{B_1}G(u)]^*\bigr) = b_+\bigl(1_{B_1}G(u^*),1_{B_1}G(u^*) \bigr) = \Phi^+(u^*)
$$
for every $u \in \cB_1$ and since that $\cB_1^* \subset \cB_1$ we have
\begin{equation}
\label{eq:schwarz-characterization}
m_1^+(N,G)= \sup_{u \in \cB_1^*} \Phi^+(u).
\end{equation}
By assumption, we have 
\begin{equation}
  \label{eq:estimate-C-upperx}
G(s) \le c_1 \frac{e^{\alpha_N |s|^{N/(N-1)}}}{(1+ |s|)^{\frac{N}{2(N-1)}}}  \qquad \text{for every $s \in \R$} 
\end{equation}
with a constant $c_1>0$. Let $u \in \cB_1^*$ and $v:= 1_{B_1} G(u)$. By Corollary~\ref{corollary-nonuniform-finiteness}, we then have $b_\pm(v)<\infty$. Therefore we have, by Corollary~\ref{cor-newtons-theorem}(ii) and (\ref{eq:estimate-C-upperx}),  
\begin{align*}
&\frac{\Phi(u)}{2\omega_{N-1}^2}=
\frac{b_0(v)}{2 \omega_{N-1}^2}= \int_{0}^{1} r^{N-1} v(r)\ln \frac{1}{r} \int_{0}^{r}\rho^{N-1}v(\rho)d\rho dr \\ &\le \int_{0}^{1} r^{N-1} c_1 \frac{e^{\alpha_N |u|^{N/(N-1)}(r)}}{(1+ |u(r)|)^{\frac{N}{2(N-1)}}}\ln \frac{1}{r} \int_{0}^{r}\rho^{N-1}c_1 \frac{e^{\alpha_N |u|^{N/(N-1)}(\rho)}}{(1+ |u(\rho)|)^{\frac{N}{2(N-1)}}}d\rho dr\\
&\le c_2 
\end{align*}
with 
$$
c_2:= c_1^2 \sup_{u \in \cB_{1,rad}}\Phi_{\beta_1,\beta_2}(u,u) < \infty,
$$
where $\cB_{1,rad}$ and $\Phi_{\beta_1,\beta_2}$ are defined in Lemma~\ref{general-boundedness} for $\beta_1= \beta_2 = -\frac{N}{2(N-1)}$. Combining this inequality with Lemma~\ref{C-B-1-restimate}, we thus deduce that
\begin{equation}
  \label{eq:restricted-ineq-l-infty}
\Phi_+(u) = \Phi(u) + \Phi_-(u) \le c_3 \quad \text{for $u \in \cB_1^*$ with}\quad 
c_3 := 2\omega_{N-1}^2 c_2 + \sup_{\cB_1} \Phi_-  < \infty.
\end{equation}
Hence $m_1^+(N, G) \le c_3 < \infty$, and therefore also $m_1(G) \le m_1^+(N, G)< \infty$, as claimed.
\end{proof}
 \begin{proposition}
\label{C-infiniteness}
Suppose that $G$ satisfies $(G_0)$ and has at least $\beta$-critical growth for some $\beta>-\frac{N}{2(N-1)}$. Then there exists a sequence of functions $u_n \in \cB_1 \cap L^\infty(B_1)$ with $\Phi(u_n) \to \infty$ as $n \to \infty$.
\end{proposition}

\begin{proof}
The assumption implies the existence of a constant $c_1>0$ with 
\begin{equation}
  \label{eq:estimate-C-lower}
G(s) \ge  c_1 s^\beta e^{\alpha_N |s|^{N/(N-1)}} \qquad \text{for $|s| \ge c_1$.} 
\end{equation}
For $n \in \N$, 
${n \geq 2}$, we now define 
 $u_n= m_n \in W^{1,N}_0(B_1) \cap L^\infty(B_1)$ as in \cite[p. 309]{Doo-Marcos}, namely
\[
m_n:=
\begin{cases}
	\frac{1}{{\omega^{1/N}_{N-1}}} {(\ln n)^{(N-1)/N}}  \quad  \  0 \leq |x| \leq \frac{1}{n}, \\
\frac{1}{{\omega^{1/N}_{N-1}}} \frac{\ln(\frac{1}{\abs{x}})}{(\ln n)^{1/N}}, \quad  \,\,\,\, \,\,\,\,\, \,\,\, \ \frac{1}{n} \leq |x| \leq 1.
\end{cases}
\]

As noted in \cite[p. 310]{Doo-Marcos}, we then have $|\nabla u_n|_N \le 1$ for $n$ large and thus $u_n \in \cB_1^*$. Moreover, $v_n :=G(u_n) \in L^\infty(B_1)$ and therefore 
$$
\Phi_\pm(u_n) = b_\pm(v_n,v_n) < \infty \qquad \text{for $n \in \N$.}
$$
By (\ref{eq:estimate-C-lower}) we have, for $n$ large,  
$$
v_n \:\ge\:  c_1 
\Bigl(\frac{1}{{\omega^{\frac{1}{N}}_{N-1}}} {(\ln n)^{\frac{N-1}{N}}}\Bigr)^{\beta} 
e^{\alpha_N\Bigl(\frac{1}{{\omega^{\frac{1}{N}}_{N-1}}} {(\ln n)^{\frac{N-1}{N}}}\Bigr)^{\frac{N}{N-1}}}  
$$

$$
=c_1 
\Bigl(\frac{\ln n}{{\omega^{\frac{1}{N-1}}_{N-1}}}\Bigr)^{\frac{N-1}{N}\beta} e^{\alpha_N\Bigl(\frac{\ln n}{{\omega^{\frac{1}{N-1}}_{N-1}}}\Bigr)^{\frac{N}{N-1}\frac{N-1}{N}} }
$$

$$
\:\ge\: 
 c_{2} (\ln n)^{\frac{N-1}{N}\beta} n^{N}
\qquad \text{on $B_{\frac{1}{n}}(0)$}
$$
with a constant $c_2>0$. We derive that
\begin{align*}
\frac{b_0(v_n,v_n)}{\omega_{N-1}^2}& \geq  2\int_{0}^{\frac{1}{n}} r^{N-1} v_n(r) \ln \frac{1}{r} \int_{0}^r \rho^{N-1}  v_n(\rho)d\rho dr\\
&\ge 2 \Bigl(c_{2} (\ln n)^{\frac{N-1}{N}\beta} n^{N}
\Bigr)^2 \int_{0}^{\frac{1}{n}} r^{N-1} \ln \frac{1}{r} \int_{0}^{r} \rho^{N-1} d\rho dr\\
&=\frac{1}{N} 2 c_2^2 (\ln n)^{2\frac{N-1}{N}\beta} n^{2N}
 \int_{0}^{\frac{1}{n}} r^{2N-1}\ln \frac{1}{r} dr
\\ &=\frac{2}{N} c_2^2 (\ln n)^{2\frac{N-1}{N}\beta} n^{2N} \frac{1}{2N n^{2N}} \Bigr( \ln n +\frac{1}{2N} \Bigl)
\\
& \ge \frac{1}{N^2} c_2^2 (\ln n)^{2\frac{N-1}{N}\beta+1},
\end{align*}
so that $b_0(v_n,v_n) \to \infty$ as $n \to \infty$ since $\beta >-\frac{N}{2(N-1)}$.  This shows that $\Phi(u_n) = b_0(v_n,v_n)  \to  \infty$ as $n \to \infty$, 
as required. 
\end{proof}
Next we wish to prove the existence of a maximizer $u \in \cB_1^*$ for $m_1(N, G)$. We first prove the following convergence result. 

\begin{proposition}
\label{case-1-conclusion}
Suppose that $G$ satisfies $(G_0)$ and has at most $\beta$-critical growth with $\beta<-\frac{N}{2(N-1)}$, and let $(u_n)_n$ be a sequence in $\cB_1^*$ with $u_n \weak 0$ in $W^{1,N}_0(B_1)$. Then 
\begin{equation}
  \label{eq:case-1-conclusion-eq}
\Phi(u_n) \to \Phi(0) \qquad \text{as $n \to \infty$.}
\end{equation}
\end{proposition}

\begin{proof}
Since $u_n \weak 0$ in $W^{1,N}_0(B_1)$ and $W^{1,N}_0(B_1)$ is compactly embedded into $L^p(B_1)$ for $N<p<\infty$, we have 
\begin{equation}
  \label{eq:l-p-strong-zero}
u_n \to 0 \qquad \text{in $L^p(B_1)$ for $N < p< \infty$.}
\end{equation}
Since $u_n \in \cB_1^*$ for every $n \in \N$, (\ref{eq:l-p-strong-zero}) implies that 
\begin{equation}
  \label{eq:l-p-strong-zero-locally-uniform}
u_n \to 0 \qquad \text{uniformly in $[\delta,1]$ for every $\delta  \in (0,1)$.}
\end{equation}
We now write $G= \kappa_0 + \tilde G$ with $\kappa_0=G(0)$, where the function $\tilde G = G - \kappa_0$ also satisfies $(G_0)$, $\tilde G(0)=0$ and has at most $\beta$-critical growth. Consequently,
\begin{equation}
  \label{eq:beta-critical-estimate}
\tilde G(t)\le c_1(1 + |t|)^\beta e^{\alpha_N |t|^{N/(N-1)}} \le c_1 e^{\alpha_N |t|^{N/(N-1)}} \qquad \text{for $t \in \R$ with a constant $c_1>0$.}
\end{equation}
With
$$
v_n:= 1_{B_1}\tilde G(u_n) \qquad \text{for $n \in \N$,}
$$
we then have
$$
\Phi(u_n)= b_0(v_n) + 2 
b_0 \bigl(1_{B_1}\kappa_0,v_n\bigr) + b_0(\kappa_0 1_{B_1})
= b_0(v_n) + 2 
b_0 \bigl(1_{B_1} \kappa_0,v_n\bigr) + \Phi(0).
$$
By Corollary~\ref{cor-newtons-theorem}(i), we have 
$$
b_0 \bigl(1_{B_1} \kappa_0,v_n \bigr) = \omega_{N-1}^2 \kappa_0 \int_{0}^1 r^{N-1} v_n(r) g(r)dr \qquad \quad \text{with}\quad 
g(r)= \ln \frac{1}{r}  \int_{0}^r \rho^{N-1} d\rho + \int_{r}^1 \rho^{N-1} (\ln \frac{1}{\rho}) d\rho. 
$$
Moreover, for any $\delta \in (0,1)$ we have, by (\ref{dis}) and (\ref{eq:beta-critical-estimate}), 
\begin{equation}
  \label{eq:g-L-infty-delta-est}
\Bigl| \int_{0}^\delta r^{N-1} v_n(r) g(r)dr \Bigr| \le \frac{c_1}{\omega_{N-1}} \|g\|_{L^\infty(0,\delta)}\int_{B_1}e^{\alpha_N t^{N/(N-1)}} dx \le \frac{c_1 c(B_1)}{\omega_{N-1}} \|g\|_{L^\infty(0,\delta)}.
\end{equation}
By (\ref{eq:l-p-strong-zero-locally-uniform}) and since $\tilde G(0)=0$, we also have
\begin{equation}
  \label{eq:delta-uniform-convergence}
v_n \to 0 \qquad \text{uniformly in $[\delta, 1]$ for every $\delta \in (0,1)$.}
\end{equation}
Combining (\ref{eq:g-L-infty-delta-est}), (\ref{eq:delta-uniform-convergence}) and the fact that $g(r) \to 0$ as $r \to 0$, we see that  
$$
b_0 \bigl(1_{B_1} \kappa_0,v_n \bigr) \to 0 \qquad \text{as $n \to \infty$.}
$$
To prove (\ref{eq:case-1-conclusion-eq}), it thus remains to show that 
\begin{equation}
  \label{limit}
b_0(v_n) \to 0 \qquad \text{as $n \to \infty$.}  
\end{equation}
To see (\ref{limit}), we note that for every $\delta \in (0,1)$ we have
$$
\frac{b_0(v_n)}{2\omega_{N-1}^2} = \int_{0}^1 r^{N-1} v_n(r) \ln \frac{1}{r} \int_{0}^r \rho^{N-1}  v_n(\rho)d\rho dr = M_n^\delta + N_n^\delta,
$$
where, by (\ref{dis}), \eqref{eq:beta-critical-estimate} and (\ref{eq:delta-uniform-convergence})
\begin{align}
M_n^\delta &:= \int_{\delta}^1 r^{N-1} v_n(r) \ln \frac{1}{r} \int_{0}^r \rho^{N-1}  v_n(\rho)d\rho dr \le c_1 \int_{\delta}^1 r^{N-1} v_n(r) \ln \frac{1}{r} \int_{0}^1 \rho^{N-1} e^{\alpha_N |u_n|^{N/(N-1)}(\rho)} d\rho dr \nonumber\\
&\le \frac{c_1 c(B_1)}{\omega_{N-1}} \int_{\delta}^1 r^{N-1} v_n(r) \ln \frac{1}{r} dr \to 0 \qquad \text{as $n \to \infty$.} \label{b-n-v-n-zero-est-1}
\end{align}
To estimate 
$$
N_n^\delta:= \int_{0}^\delta r^{N-1} v_n(r) \ln \frac{1}{r} \int_{0}^r \rho^{N-1}  v_n(\rho)d\rho dr 
$$
we fix $\eps \in (0,\frac{1}{\alpha_N})$ and define, for any $n \in \N$,
$$
A_n^+
= \{r \in (0,1]\::\: u_n(r) \ge  [\epsilon(-\ln r)]^{\frac{N-1}{N}} \},\qquad
A_n^-
:= \{r \in (0,1]\::\: u_n(r) < [\epsilon(-\ln r)]^{\frac{N-1}{N}} \}.
$$
By \eqref{eq:beta-critical-estimate}, we have 
\begin{equation}
\label{eq:A--estimate8}
v_n(r) \le c_1 e^{\alpha_N |u_n|^{\frac{N}{N-1}}(r)} \le  \frac{c_1}{r^{\alpha_N \eps}}
\qquad \text{for $r \in A_n^-$, and}
\end{equation}
\begin{equation}
\label{eq:A-+estimate0}
v_n(r) \le c_1 \bigl(1+ [\epsilon(-\ln r)]^{\frac{N-1}{N}}\bigr)^\beta e^{\alpha_N u_n^{\frac{N}{N-1}}(r)} \qquad \text{for $r \in A_n^+$.}
\end{equation}
We now write 
\begin{align*}
N_n^\delta= \int_{A_n^- \cap (0,\delta)} r v_n(r) \ln \frac{1}{r}
\int_{0}^r \rho^{N-1}  v_n(\rho)d\rho dr + \int_{A_n^+ \cap (0,\delta)} r^{N-1} v_n(r) \ln \frac{1}{r}\int_{0}^r \rho^{N-1}  v_n(\rho)d\rho dr,
\end{align*}
where, by (\ref{dis}) and (\ref{eq:A--estimate8}), 
\begin{align*}
&\int_{A_n^- \cap (0,\delta)} r^{N-1} v_n(r) \ln \frac{1}{r}
\int_{0}^r \rho^{N-1}  v_n(\rho)d\rho dr  \le c_1^2 \int_{0}^\delta r^{N-1- \alpha_N \eps} \ln \frac{1}{r} \int_{0}^1 \rho^{N-1} e^{\alpha_N |u_n|^{N/(N-1)}(\rho)} d\rho dr \\
&\le  \frac{c_1^2 c(B_1)}{\omega_{N-1}} \int_{0}^\delta r^{N-1-\alpha_N \eps} \ln \frac{1}{r} dr < +\infty
\label{N-delta-1-est}
\end{align*}
for all $n \in \N$. Moreover, we have 
\begin{align*}
&\int_{A_n^+ \cap (0,\delta)} r^{N-1} v_n(r) \ln \frac{1}{r}\int_{0}^r \rho^{N-1}  v_n(\rho)d\rho dr  \\  
&= \int_{A_n^+ \cap (0,\delta)} r^{N-1} v_n(r) \ln \frac{1}{r}\int_{A_n^+\cap (0,r)} \rho^{N-1}  v_n(\rho)d\rho dr +\\
&+ \int_{A_n^+ \cap (0,\delta)} r^{N-1} v_n(r) \ln \frac{1}{r}\int_{A_n^- \cap (0,r)} \rho^{N-1}  v_n(\rho)d\rho dr,
\end{align*}
where 
\begin{align*}
&\int_{A_n^+ \cap (0,\delta]} r^{N-1}v_n(r) \ln \frac{1}{r} \int_{A_n^+ \cap [0,r]} \rho^{N-1}  v_n(\rho)d\rho dr \nonumber\\
&\le c_1^2 \int_{A_n^+ \cap (0,\delta]} 
r^{N-1} \bigl(1+ 
[\eps(-\ln r)]^{(N-1)/N}\bigr)^\beta e^{\alpha_N |u_n|^{N/(N-1)}(r)}
  \ln \frac{1}{r} \times \\ 
& \times \int_{A_n^+ \cap [0,r]} \rho^{N-1} \bigl(1+ 
[\eps(-\ln \rho)]^{(N-1)/N}\bigr)^\beta e^{\alpha_N |u_n|^{N/(N-1)}(\rho)} d\rho dr \nonumber\\
&\le c_1^2 \int_{A_n^+ \cap (0,\delta]} (-\ln r)\bigl(1+ 
(\eps(-\ln r))^{(N-1)/N}\bigr)^{2\beta} r^{N-1} e^{\alpha_N |u_n|^{(N-1)/N}(r)}\int_0^r  
\rho^{N-1} e^{\alpha_N |u_n|^{(N-1)/N}(\rho)}d\rho dr \\
&\le \frac{c_1^2}{\eps}\int_{A_n^+ \cap (0,\delta]} \bigl(1+ 
(\eps(-\ln r))^{(N-1)/N}\bigr)^{\frac{N}{N-1}}\bigl(1+ 
(\eps(-\ln r))^{(N-1)/N}\bigr)^{2\beta} r^{N-1} e^{\alpha_N |u_n|^{(N-1)/N}(r)}  \\& \\&\times  \int_0^r  
\rho^{N-1} e^{\alpha_N |u_n|^{(N-1)/N}(\rho)}d\rho dr
\\&\le \frac{c_1^2}{\omega_{N-1}^2\eps} \bigl(1+ 
(-\ln \delta)^{(N-1)/N}\bigr)^{N/(N-1)+2\beta} \Bigl(\int_{0}^1
r^{N-1} e^{\alpha_N |u_n|^{N/(N-1)}(r)} dr\Bigr)^2
\\ & \le \frac{\bigl(c_1 c(B_1)\bigr)^2}{\eps\omega_{N-1}^2} (\eps(-\ln \delta))^{(N-1)/N}\bigr)^{N/(N-1)+2\beta}
\end{align*}
again by (\ref{dis}). Here we used the assumption $\beta<-\frac{N}{2(N-1)}$. Finally, we have 
\begin{align}
&\int_{A_n^+ \cap [0,\delta]}
r^{N-1} v_n(r) \ln \frac{1}{r} \int_{A_n^- \cap [0,r]} \rho^{N-1}  v_n(\rho)d\rho dr\\
& \leq 
c_1^2 \int_{A_n^+ \cap [0,\delta]} {r^{N-1} e^{\alpha_N |u_n|^{N/(N-1)}(r)}}
\ln \frac{1}{r}  \int_{0}^r 
\rho^{N-1- \alpha_N \eps} d\rho dr  \\
&\le \frac{c_1^2}{N-\alpha_N \eps} \int_{0}^\delta r^{2N-1-\alpha_N \eps} e^{\alpha_N |u_n|^{N/(N-1)}(r)} \ln \frac{1}{r} \,dr \\
&\leq c_1^2 
C_\delta \int_{0}^1 r^{N-1} e^{\alpha_N |u_n|^{N/(N-1)}(r)}\,dr \le \frac{c_1^2 
C_\delta c(B_1)}{\omega_{N-1}} 
\label{N-delta-3-est}
\end{align}
by (\ref{dis}) with $C_\delta = \sup \limits_{r \in [0,\delta]}r^{N-\alpha_N\eps}\ln \frac{1}{r}$. Since, as $\eps \in (0,\frac{1}{\alpha_N})$ and $\beta<-\frac{N}{2(N-1)}$,  we infer that
$$
\lim_{\delta \to 0}\sup_{n \in \N}N_n^\delta =0.
$$
Combining this with (\ref{b-n-v-n-zero-est-1}), we infer (\ref{limit}), as claimed.
\end{proof}

\begin{proposition}
\label{C-existence-of-maximizer}
Suppose that $G$ satisfies $(G_0)$ and has at most $\beta$-critical growth with $\beta<-\frac{N}{2(N-1)}$. Then the value 
$m_1(N, G)< \infty$ is attained by a function $u \in \cB_1^*$. 
\end{proposition}

\begin{proof}
Let $(u_n)_n$ be a maximizing sequence in $\cB_1$ for $m_G$. By Lemma~\ref{Riesz-rearrangement}, we may assume that $u_n \in \cB_1^*$ for $n \in \N$. Since $\cB_1$ is bounded in $W^{1,N}_0(B_1)$, we may also assume that  
\begin{equation}
  \label{eq:b-1-weak-convergence}
u_n \weak u \in W^{1,N}_0(B_1) \qquad \text{with $u \in \cB_1^*$.}
\end{equation}
By \cite[Theorem 1.6]{Lions}, we have two possibilities. Either 
\smallskip
\begin{itemize}
\item[i)] $u=0$, or 
\item[ii)] $u \not = 0$, and $\int_{B_1}  e^{(\alpha_N +t )u_n^{N/(N-1)} } \ dx$ is bounded for some $t>0$ and thus 
  \begin{equation}
    \label{eq:ii-l-1-convergence}
e^{\alpha_N u_n^{N/(N-1)}} \to e^{\alpha_N u^{N/(N-1)} }  \quad \hbox{in}  \ \ L^1(B_1).
  \end{equation}
\end{itemize}
Assume first that $i)$ holds. In this case Proposition~\ref{case-1-conclusion} implies that 
\begin{equation*}
m_1(N, G) = \lim_{n \to \infty} \Phi(u_n) = \Phi(0), 
\end{equation*}
which is impossible by Remark~\ref{m-F-greater-F-0}. So this case does not occur.\\
It remains to consider the case where $(u_n)_n$ satisfies $ii)$. 
Set $v_n:= 1_{B_1} G(u_n)$ for $n \in \N$ and $v:= 1_{B_1}G(u)$. By $ii)$,  $\int_{B_1}  e^{(\alpha_N +t ) u_n^{N/(N-1)}}\ dx$ is bounded for some $t>0$ and thus  $v_n$ is bounded in  $L^{s_0}(\R^N)$ with $s_0=1 + \frac{t}{\alpha_N}>1$. Moreover, since 
$$
v_n \to v\qquad \text{in $L^1(\R^N)$,}
$$
interpolation yields that 
$$
v_n \to v  \qquad \text{in $L^s(\R^N)\qquad$ for $1 \le s < s_0$.}
$$
Moreover,  
\begin{align*}
&\frac{\Phi(u_n)-\Phi(u)}{2\omega_{N-1}^2}= \int_{0}^1 r^{N-1} v_n(r) \ln \frac{1}{r} \int_{0}^r \rho^{N-1}  v_n(r) d\rho dr - \int_{0}^1 r^{N-1} v(r) \ln \frac{1}{r} \int_{0}^r \rho^{N-1}  v(\rho) d\rho dr\\
&= \int_{0}^1 r^{N-1} v_n(r) \ln \frac{1}{r} \int_{0}^r \rho^{N-1}  [v_n(\rho)-v(\rho)] d\rho dr + \int_{0}^1 r^{N-1} [v_n(r)-v(r)] \ln \frac{1}{r} \int_{0}^r \rho^{N-1}  v(\rho) d\rho dr
\end{align*}
where, for fixed $s \in (1,s_0)$,  
\begin{align*}
&\Bigl|\int_{0}^1 r^{N-1} v_n(r) \ln \frac{1}{r} \int_{0}^r \rho^{N-1}  [v_n(\rho)-v(\rho)] d\rho dr\Bigr|\le \frac{|v_n-v|_{s}}{\omega_{N-1}} \int_{0}^1 r^{N-1} |B_r|^{\frac{1}{s'}}  v_n(r) \ln \frac{1}{r}  dr  \\
&=\frac{|v_n-v|_{s}}{\omega_{N-1}} \int_{0}^1 r^{N-1} \bigg|\frac{r^N \pi^{\frac{N}{2}}}{\Gamma(\frac{N}{2}+1 )}\bigg|^{\frac{1}{s'}}  v_n(r) \ln \frac{1}{r}  dr\\
&\le \frac{|v_n-v|_{s} \pi^{\frac{N}{2s'}}}{\Gamma(\frac{N}{2}+1)\omega_{N-1}} 
\int_{0}^1 r^{N-1+\frac{N}{s'}} v_n(r) \ln \frac{1}{r} dr \le C |v_n-v|_{s} |v_n|_{1}  \to 0 \qquad \text{as $n \to \infty$ } 
\end{align*}
with $C:= \frac{\pi^{{\frac{N}{2s'}}}}{\Gamma(N/2+1)\omega_{N-1}^2}\sup \limits_{r \in (0,1]} r^{\frac{N}{s'}} \ln \frac{1}{r}$  and also 
\begin{align*}
&\Bigl|\int_{0}^1 r^{N-1} [v_n(r)-v(r)] \ln \frac{1}{r} \int_{0}^r \rho^{N-1}  v(\rho) d\rho dr\Bigr| \le \frac{|v|_{s}}{\omega_{N-1}} \int_{0}^1 r^{N-1} |B_r|^{\frac{1}{s'}}  [v_n(r)-v(r)] \ln \frac{1}{r}  dr  \\
&\le \frac{|v|_{s}\pi^{\frac{N}{2s'}}}{\Gamma(\frac{N}{2}+1)\omega_{N-1}} 
\int_{0}^1 r^{N-1+\frac{N}{2s'}} [v_n(r)-v(r)] \ln \frac{1}{r} dr \le C |v|_{s} |v_n-v|_{1} \to 0 \qquad \text{as $n \to \infty$.} 
\end{align*}
We thus conclude that 
\begin{equation*}
m_1(N, G) = \lim_{n \to \infty} \Phi(u_n) = \Phi(u), 
\end{equation*}
so $m_1(N,G)$ is attained at $u \in \cB_1^*$.
\end{proof}

The proof of Theorem~\ref{sec:introduction-main-thm-C-F2} is now completed by the following lemma.

\begin{lemma}
\label{maximizer-strictly-positive}
Let $G$ satisfy $(G_0)$, and let $u \in \cB_1$ be a maximizer for $\Phi \big|_{\cB_1}$. Then, up to a change of sign, we have $u \in \cB_1^*$, and 
$u$ is strictly positive in $B_1$.  
\end{lemma}

\begin{proof}
We already know that $u \not \equiv 0$. We first assume that $u \in \cB_1^*$, and we suppose by contradiction that there exists $\tau \in (0,1)$ with $u(r)>0$ for $r \in (0,\tau)$ and $u(r)= 0$ for $r \in [\tau,1)$. Let then $\tilde u \in W^{1,N}_0(B_1)$ be defined by $\tilde u(r)= u(\tau r)$. Then we have $|\nabla \tilde u|_N = |\nabla u|_N$ and therefore $\tilde u \in \cB_1^*$. Moreover, with $v:= G(u)$ we have 
$$
\Phi(\tilde u) = 2\omega_{N-1}^2 \int_0^1 r^{N-1} v(\tau r) \ln \frac{1}{r}\int_0^r \rho^{N-1} v (\tau \rho) d\rho dr, 
$$
where 
$$
v(\tau r) \ge v(r)\quad \text{for $r \in (0,\tau)$}\qquad \text{and}\qquad v(\tau r)= G(u(\tau r)) >G(0)=v(r)\quad \text{for $r \in [\tau,1)$,}
$$
since $G$ is strictly increasing on $[0,\infty)$ by assumption $(G_0)$. It thus follows that 
$$
\Phi(\tilde u) > 2\omega_{N-1}^2 \int_0^1 r^{N-1} v(r) \ln \frac{1}{r}\int_0^r \rho^{N-1} v (\rho) d\rho dr = \Phi(u), 
$$
contrary to the maximizing property of $u$. We thus conclude that $u$ is strictly positive in $B_1$.\\
Next, we let $u \in \cB_1$ be a general maximizer of $\Phi \big|_{\cB_1}$. By Lemma~\ref{Riesz-rearrangement}(iii), it then follows that $u^* \in \cB_1^*$ is also a maximizer of $\Phi \big|_{\cB_1}$. Moreover, since $G(u) \in L^2(\R^N)$ by Lemma~\ref{lemma-nonuniform-finiteness-convolution} and $b_\pm(G(u),G(u))<\infty$, it also follows from Lemma~\ref{Riesz-rearrangement}(iii) that $G(u)$ equals $G(u)^*$ up to translation. Since $G$ is even and strictly increasing on $[0,\infty)$ by assumption $(G_0)$, this implies that $u$ equals $u^*$ up to sign and translation. Since $u \in W^{1,N}_0(B_1)$ 
and, as we have proved above, $u^*>0$ in $B_1$, it follows that $u \equiv u^* \in \cB_1^*$ or $-u \equiv u^* \in \cB_1$. The proof is thus finished.
\end{proof}
$\bigskip$

\section{Maximization problem on the entire space $\R^N$}
\label{sec:maxim-probl-entire}
In this section we complete the proof of Theorem~\ref{sec:introduction-main-thm}. We recall that 
$\cB_{\infty}:= \{u \in W^{1,N}(\R^N)\::\: \|u\|_N\le 1\}$, where $\|u\|_N^N:= |\nabla u|_N^N + |u|_N^N$. We let 
$$
\cB_\infty^*:=  \{u^* \::\: u \in \cB_{\infty}\}
$$
denote the corresponding Schwarz symmetrized set. By Lemma~\ref{symmetrized-invariance-H-1}, $\cB_\infty^*$ is a subset of $\cB_\infty$.

Throughout this section, we assume that $G: \R \to \R$ satisfies assumption $(G_1)$.
\begin{lemma}
\label{sec:maxim-probl-entire-1}
Suppose that $G$ has at most $\beta$-critical growth for some $\beta \le -\frac{N}{2(N-1)}$. Then 
$$
m_\infty^+(N, G) = \sup_{\cB_\infty}\Psi_+ <\infty.
$$
\end{lemma}

\begin{proof}
By assumption, there exists a constant $c_1>0$ with 
	\begin{equation}
	\label{eq:estimate-ball-upper}
	G(s) \le \widetilde G(s):= c_1 \frac{e^{\alpha _N |s|^{\frac{N}{N-1}}}}{(1+ \abs{s}^{\frac{N}{N-1}})^{\frac{N}{2(N-1)}}} \qquad \text{for {$s \in \R$}.} 
	\end{equation}
	We note that $\widetilde G$ satisfies assumption $(G_0)$.
        Moreover, we note that 
	\begin{equation}
	\label{eq:C-s-ineq}
	\widetilde G((t+ s)^{\frac{(N-1)}{N}})=  c_1 \frac{e^{\alpha _N(t+s)}}{(1+t+s)^{\frac{N}{2(N-1)}}}\le  c_1
	e^{\alpha _N s}\frac{e^{\alpha _N t}}{(1+t)^{\frac{N}{2(N-1)}}}= e^{\alpha _N s} \widetilde G((t)^{\frac{(N-1)}{N}})\quad \text{for $t,s \ge 0$.}
	\end{equation}
Since $\Psi_+(u^*) \ge \Psi_+(u)$ for every $u \in W^{1,N}(\R^N)$ by Lemma~\ref{Riesz-rearrangement}(i), it now suffices to consider functions $u \in \cB_\infty^*$.  We first note, by the Radial Lemma \cite[Lemma A.IV]{BL} (see also \cite{cao}), we have 
\begin{equation}
	\label{eq:radial-lemma-inequalities}
	u(r) \le  \frac{C_N}{r}|u|_N\le \frac{C_N}{r}
\end{equation}
where $C_N:=\big( \frac{N}{\omega_{N-1}}\big)^{1/N}$.
Let $v= G(u)$. We write 
$$
v= G(u)= G(u 1_{\R^N \setminus B_{1}}) + G(u 1_{B_{1}})= : v_1 + v_2. 
$$
Since $G$ is nonnegative and $G(t)=O(|t|)$ as $t \to 0$, it follows that 
$$
0 \le v(r) \le c_2 |u(r)|
\qquad \text{for $r \ge 1$}
$$
with a constant $c_2>0$ and therefore 
\begin{equation}
  \label{eq:sigma-est}
|v_1|_N \le c_2|u|_N \le c_2\|u\|_N =: c_3. 
\end{equation}

To estimate $v_2$, we now consider the radial function  
$$U= [u - u(1)]_+ \in  W^{1,N}(B_1)$$

As  in \cite[Lemma 1]{Doo-Marcos}, since the function
$h(t) :]0, \infty [ \to \R$  given by 
$h(t)= [(t+1)^{\frac{N}{N-1}} - t^{\frac{N}{N-1}} -1 ] / t^{\frac{1}{N-1}}$ is bounded, 
there esists $A(N)>0$ such that
\begin{equation}\label{tom}
	|u|^{\frac{N}{N-1}} = |U + u(1)|^{\frac{N}{N-1}} \leq U^{\frac{N}{N-1}} +
	u^{\frac{N}{N-1}}(1)  +  A \, U^{\frac{1}{N-1}}\, u(1).
\end{equation}

Again as in  \cite[Lemma 1]{Doo-Marcos}, we deduce that  for any $\sigma >0$ 
\begin{equation}\label{tom}
	 U^{\frac{1}{N-1}} \, u(1) =
	[U^{\frac{N}{N-1}}]^{\frac{1}{N}} \, 
	[u^{\frac{N}{N-1}}(1)]^{\frac{N-1}{N}}   \leq 
	\frac{\sigma}{A}
	U^{\frac{N}{N-1}}
	+
	(\frac{\sigma}{A})^{\frac{1}{(1-N)}}
	u^{\frac{N}{N-1}}(1)
\end{equation}

Therefore from (\ref{tom}) we have 
\begin{align}\label{keyab}
|u|^{\frac{N}{N-1}} & \leq U^{\frac{N}{N-1}} +
	u^{\frac{N}{N-1}}(1)   
	+  
	\sigma \, U^{\frac{N}{N-1}}
	+
	A (\frac{\sigma}{A})^{\frac{1}{1-N}}
	u^{\frac{N}{N-1}}(1)  \nonumber \\ &=
	(1 + \sigma) U^{\frac{N}{N-1}} +
\bigl(1 +A (\frac{\sigma}{A})^{\frac{1}{(1-N)}}\bigr)
	u^{\frac{N}{N-1}}(1).
\end{align}
Now choose 
$
\sigma =  \bigl(1 + \frac{u^{N}(1)}{C_N^{N}}\bigr)^\frac{1}{N-1} -1
$ 
%we have 
%$
%(1 +\epsilon)^{N-1} =  1 + %\frac{u^{N}(1)}{C_N^{N}}
%$ 
and  $W= ( 1 + \frac{u^{N}(1)}{C_N^{N}})^{\frac{1}{N}}U \in W^{1,N}_0(B_1)$. By the Radial Lemma
$$
\int_{B_1}|\nabla W|^N dx = ( 1 + \frac{u^{N}(1)}{C_N^{N}})
 \int_{B_1}|\nabla u|^N  \leq  
 (1+ \frac{C_N^N |u|_N^N}{C_N^N}) 
 \int_{B_1}|\nabla u|^N dx \leq |\nabla u|^N_N + |u|^N_N \leq 1
 $$
and thus  $U \in \mathcal{B}^*_1$, which implies in particular that
\begin{equation}
	\label{eq:l-1-norm-v-2-finite}
	|v_2|_1 \le \tilde c_1 \int_{B_1}e^{\alpha_N W^{N/(N-1)}}\,dx \le \tilde c_1 c(B_1)=:c_4 
\end{equation}
by the Trudinger-Moser inequality (\ref{dis}).

Consequently, by (\ref{eq:C-s-ineq}) and 
(\ref{keyab})
we have 
$$
v_2 \le  \tilde G((u^{N/(N-1)})^{\frac{N-1}{N}}) \le c_1 e^{\alpha_N
	[1 +	A(\frac{\sigma}{A})^{\frac{1}{(1-N)}}]u^{\frac{N}{N-1}}(1)} \,
	 \widetilde G(W) \qquad \text{in $B_{1}$}.
$$

We notice that the function $f(t): ]0, C_N ] \to \R$  defined by $f(t):= \Big(1+ \frac{A^{\frac{N}{N-1}}}{\Big(\Big(1+ \frac{t^N}{C_N^N}\Big)^{\frac{1}{N-1}}-1 \Big)^{\frac{1}{N-1}}} \Big) t^{\frac{N}{N-1}}$
	is uniformly bounded. Indeed $   \lim\limits_{t \to 0^+} f(t) = (N-1)^{\frac{1}{N-1}} (C_N A)^{\frac{N}{N-1}}.$
Then there exists $K>0$, depending on $N$ and $A$ (independent of $u$) such that\begin{equation}\label{otto}
    f(u(1))\leq K.
\end{equation}
We now have 
\begin{equation}
  \label{eq:Phi-plus-splitting}
\Psi_+(u)= b_+(v)= b_+(v_1) + b_+(v_2)+ 2 b_+(v_1,v_2)
\end{equation}
where	
\begin{equation*}
    b_+(v_2)=\int_{\R^N}\int_{\R^N}\ln^+\frac{1}{|x-y|}v_2(x)v_2(y)dxdy\le c_1^2
     e^{2\alpha_N [1 +	A(\frac{A}{\sigma})^{\frac{1}{(N-1)}}]u^{\frac{N}{N-1}}(1)} \ b_+(\widetilde G(W),\widetilde G(W))
\end{equation*}
$$
\le c_1^2 e^{2\alpha_N K}
	m_1^+(\widetilde G)= :c_5 <\infty
$$
by Proposition~\ref{C-finiteness} and 
(\ref{otto}).

 Moreover, by (ii) of Lemma \ref{lemma-nonuniform-finiteness-convolution}, Young's inequality, (\ref{eq:sigma-est}), (\ref{eq:l-1-norm-v-2-finite}), (\ref{otto}),
  we have 
$$
b_+(v_1,v_2)=\int_{\R^N}\int_{\R^N}\ln^+\frac{1}{|x-y|}v_1(x)v_2(y)dxdy \le |v_1|_2 
\Bigl|\ln^+ \frac{1}{|\cdot|}\Bigr|_{2} |v_2|_1 \le c_3 \Bigl|\ln^+ \frac{1}{|\cdot|}\Bigr|_{2}  c_4 =:c_6
$$
Moreover, again by (ii) of Lemma \ref{lemma-nonuniform-finiteness-convolution} and Young's inequality,
$$
b_+(v_1) \le |v_1|_{2}^2 \Bigl|\ln^+ \frac{1}{|\cdot|}\Bigr|_1 \le c_3^2 \Bigl|\ln^+ \frac{1}{|\cdot|}\Bigr|_1  = :c_7.
$$
Inserting these estimates in (\ref{eq:Phi-plus-splitting}), we deduce that $\Psi_+(u) \le c_5 + c_7 + 2c_6<\infty$. Thus the claim is proved.
\end{proof}
\begin{proposition}
	\label{existence-maximizer}
	Suppose that $G$ satisfies $(G_1)$ and that $G$ has at most $\beta$-critical growth for some $\beta<-\frac{N}{2(N-1)}$. Then the value 
$m_\infty(G)< \infty$ is attained by a function $u \in \cB_\infty^*$. 
\end{proposition}

\begin{proof}
   Let $(u_n)_n$ be a maximizing sequence in $\cB_\infty$ for $m_\infty(N, G)$. By Lemma~\ref{Riesz-rearrangement}, we may assume that $u_n \in \cB_\infty^*$ for $n \in \N$. Moreover, we have $b_-(G(u_n))< \infty$ for every $n \in \N$. Hence, by Corollary~\ref{cor-newtons-theorem}, we have 
	$$
	\Psi(u_n)= b_0(G(u_n),G(u_n))=b_+(G(u_n))-b_-(G(u_n))=2 \omega_{N-1}^2 \Bigl(\Psi_2(u_n) -\Psi_1(u_n)\Bigr)\qquad \text{for $n \in \N$,}
	$$
with the nonnegative functionals $\Psi_1, \Psi_2: \cB_\infty^* \to [0,\infty)$ given by 
$$
	\Psi_1(u) = \int_{1}^\infty r^{N-1} G(u(r)) \ln r \int_{0}^r \rho^{N-1}  G(u(\rho))d\rho dr,
 $$
 
 $$\Psi_2(u)= \int_{0}^1 r^{N-1} G(u(r)) \ln \frac{1}{r} \int_{0}^r \rho^{N-1}  G(u(\rho))d\rho dr.
	$$
Since $\cB_\infty$ is bounded in $W^{1,N}(\R^N)$, we may assume that  
\begin{equation}
  \label{eq:b-1-weak-convergence-r-2}
u_n \weak u \in W^{1,N}(\R^N) \qquad \text{with $u \in \cB_\infty^*$.}
\end{equation}
We first claim that 
\begin{equation}
  \label{eq:phi-2-convergence}
\Psi_2(u) \ge  \limsup_{n \to \infty}\Psi_2(u_n).
\end{equation}
By \cite[Theorem 1.6]{Lions}, we have two possibilities. Either 
\smallskip
\begin{itemize}
\item[i)] $u=0$, or 
\item[ii)] $u \not = 0$, and $\int_{B_1}  e^{(\alpha_N +t )u_n^{\frac{N}{N-1}} } \ dx$ is bounded for some $t>0$ and thus 
  \begin{equation}
    \label{eq:ii-l-1-convergence-r-2}
e^{(\alpha_N +t )u_n^{\frac{N}{N-1}}} \to e^{(\alpha_N +t )u^{\frac{N}{N-1}}}  \quad \hbox{in}  \ \ L^1(B_1).
  \end{equation}
\end{itemize}
We first assume that i) holds. As in the proof of Proposition~\ref{case-1-conclusion}, we then deduce that 
\begin{equation}
  \label{eq:un-r-converge-zero-pointwise}
u_n(r) \to 0 \qquad \text{as $n \to \infty$ for all $r>0$.}  
\end{equation}
Moreover, there exists a constant $c>0$ with 
	\begin{equation}
	\label{eq:estimate-C-upperx-r-2}
	G(s) \le \widetilde G(s):= \frac{c}{\alpha_N} (1+ \alpha_N|s|^{\frac{N}{N-1}})^{\frac{\beta}{N}(N-1)}e^{\alpha_N s^{\frac{N}{N-1}}}\qquad \text{for {$s \in \R$}.} 
	\end{equation}
We may assume that $-\frac{N}{N-1} < \beta<-\frac{N}{2(N-1)}$ from now on. Then 
$$
{\widetilde G}'(s) =\frac{c}{\alpha_N} \frac{\beta(N-1)}{N} (1+ \alpha_N s^{\frac{N}{N-1}})^{\frac{\beta}{N}(N-1)-1} \alpha_N \frac{N}{N-1} s^{\frac{N}{N-1}-1}e^{\alpha_N s^{\frac{N}{N-1}}}$$
$$
+\frac{c}{\alpha_N}  (1+ \alpha_N s^{\frac{N}{N-1}})^{\frac{\beta}{N}(N-1)}\alpha_N \frac{N}{N-1} s^{\frac{N}{N-1}-1}e^{\alpha_N s^{\frac{N}{N-1}}}
$$

$$=c (1+ \alpha_N s^{\frac{N}{N-1}})^{\frac{\beta}{N}(N-1)-1} s^{\frac{N}{N-1}-1}e^{\alpha_N s^{\frac{N}{N-1}}}\frac{N}{N-1} \times
$$
$$\times \Bigl[\beta \frac{N-1}{N} +1+ \alpha_N s^{\frac{N}{N-1}-1}\Bigr]> 0
\qquad \text{for $s >0$}
$$
and therefore $\widetilde G$ satisfies assumption $(G_0)$. Let
\[ W_n:=\bigl(1 + \frac{u_n^{N}(1)}{C_N^{N}}\bigr)^\frac{1}{N} 
(u_n(\cdot )-u_n(1))_+ \in \cB^*_1.\]
Arguing as in the proof of Lemma \ref{sec:maxim-probl-entire-1}, there exists $A=A(N)$ such that for any $n \in \N$ we have \[ u_n^{\frac{N}{N-1}} \le W_n^{\frac{N}{N-1}}+ (1+A^{\frac{N}{N-1}}\sigma_n^{\frac{1}{1-N}})u_n^{\frac{N}{N-1}}(1)\]  
where $\sigma_n =  \bigl(1 + \frac{u_n^{N}(1)}{C_N^{N}}\bigr)^\frac{1}{N-1} -1.$ 
Setting $v_n:= G(u_n)$, taking into account $(\ref{otto})$, there exists $K>0$ such that
\begin{align*}
	v_n &= G\Bigl((u_n^{\frac{N}{N-1}})^\frac{N-1}{N}\Bigr) \\  
	& \le c e^{\alpha_N \Bigl((1+A^{\frac{N}{N-1}}\sigma_n^{\frac{1}{1-N}})u_n^{\frac{N}{N-1}}(1)\Bigr)} \widetilde G(W_n) \leq c e^{\alpha_N K} \widetilde G(W_n)  \qquad \text{for $n \in \N$}.
\end{align*}
For given $\varepsilon>0$,  we have by (\ref{eq:un-r-converge-zero-pointwise}) and since $G$ is strictly increasing on $[0,\infty)$, 
\begin{align*}
	v_n &= G(u_n) = G\Bigl(\bigl(1 + \frac{u_n^{N}(1)}{C_N^{N}}\bigr)^{-\frac{1}{N}} W_n + u_n(1)\Bigr) \\
	&\le G\Bigl(W_n+ u_n(1)\Bigr) \le G\Bigl(W_n+ \varepsilon\Bigr)  \quad \text{in $B_1$ for $n$ sufficiently large}
\end{align*} by (\ref{eq:un-r-converge-zero-pointwise}). We now fix $\eps>0$ and define $\varepsilon>0$ and define $G_\varepsilon: \R \to \R$ by $G_\varepsilon(t)=\min \{ G( |t|+\varepsilon ),c e^{\alpha_N K } \widetilde G(|t|) \}$. Then $G_\eps$ has at most $\beta$-critical growth, satisfies $(G_0)$.   
Moreover, we have $v_n \le G_\eps(U_n)$ and therefore
\begin{align*}
	\Psi_2(u_n) \le \int_{0}^1 r^{N-1} v_n(r) \ln \frac{1}{r} \int_{0}^r \rho^{N-1}  
	v_n(\rho) d\rho dr &\le \int_{0}^1 r^{N-1} {G_\eps}(U_n(r)) \ln \frac{1}{r} \int_{0}^r \rho^{N-1}  
	{G_\eps}(U_n(\rho)) d\rho dr\\
	&= \frac{1}{2\omega_{N-1}^2}b_0(1_{B_1} G_\eps(U_n),1_{B_1} G_\eps(U_n))
\end{align*}
for $n$ sufficiently large. We now define 
$$
\Phi_\eps: \cB_1 \to \R, \qquad \Phi_\eps(U) = b_0(1_{B_1} G_\eps(U),1_{B_1} G_\eps(U))
$$
Then Proposition~\ref{case-1-conclusion} applies to $\Phi_\eps$ and yields that
$$
\Psi_2(u_n) \le \Phi_\eps(U_n) \to \Phi_\eps(0) \qquad \text{as $n \to \infty$,}
$$
since $U_n \weak 0$ in $W^{1,N}_0(B_1)$. Moreover, since $G(0)=0$ by assumption, for $\eps>0$ sufficiently small we have  
$$
\Phi_\eps(0)= b_0(1_{B_1} G_\eps(0),1_{B_1} G_\eps(0))= b_0(1_{B_1} G(\eps),
1_{B_1} G(\eps)).
$$
Whereas 
$$
b_0(1_{B_1} G(\eps),
1_{B_1} G(\eps)) \to b_0 (1_{B_1} G(0),
1_{B_1} G(0))= b_0(0,0)= 0 \qquad \text{as $\eps \to 0^+$.}
$$
It thus follows that 
$$
\limsup_{n \to \infty}\Psi_2(u_n) \le 0 = \Psi_2(0) = \Psi_2(u)
$$
and therefore (\ref{eq:phi-2-convergence}) holds in this case.\\
We now assume that alternative (ii) holds. With $v_n= G(u_n)$ and $v= G(u)$, we then deduce, as in the proof of Proposition~\ref{C-existence-of-maximizer}, that  $v_n$ is bounded in  $L^{s_0}(\R^N)$ with $s_0=1 + \frac{t}{\alpha_N}>1$, and that
$$
v_n \to v  \qquad \text{in $L^s(\R^N)\qquad$ for $1 \le s < s_0$.}
$$
Moreover, since 
$$
\Psi_2(u_n)= \int_{0}^1 r^{N-1} v_n(r) \ln \frac{1}{r} \int_{0}^r \rho^{N-1}  v_n(r) d\rho dr
$$
and therefore 
\begin{align*}
&\Psi_2(u_n)-\Psi_2(u)= \int_{0}^1 r^{N-1} v_n(r) \ln \frac{1}{r} \int_{0}^r \rho^{N-1}  v_n(r) d\rho dr - \int_{0}^1 r^{N-1} v(r) \ln \frac{1}{r} \int_{0}^r \rho^{N-1}  v(\rho) d\rho dr\\
&= \int_{0}^1 r^{N-1} v_n(r) \ln \frac{1}{r} \int_{0}^r \rho^{N-1}  [v_n(\rho)-v(\rho)] d\rho dr + \int_{0}^1 r^{N-1} [v_n(r)-v(r)] \ln \frac{1}{r} \int_{0}^r \rho^{N-1}  v(r) d\rho dr
\end{align*}
we may argue exactly as in the proof of Proposition~\ref{C-existence-of-maximizer} to see that $\Psi_2(u_n)-\Psi_2(u) \to 0$ as $n \to \infty$. 
Hence (\ref{eq:phi-2-convergence}) also holds in this case.\\
Finally, by Fatou's lemma, we also have 
$\Psi_1(u) \le \liminf \limits_{n \to \infty}\Psi_1(u_n)$. Combining this with (\ref{eq:phi-2-convergence}), we conclude that $\Psi(u)\ge \limsup \limits_{n \to \infty}\Psi(u_n)$. Hence $u$ is a maximizer of $\Psi$ in $\cB_\infty$. 
\end{proof}
\begin{proposition}
\label{B-infty-infiniteness}
Suppose that $G$ satisfies $(G_0)$ and has at least $\beta$-critical growth for some $\beta>-\frac{N}{2(N-1)}$. Then there exists a sequence of functions $u_n \in \cB_\infty \cap L^\infty(\R^N)$ with $\Psi(u_n) \to \infty$ as $n \to \infty$.
\end{proposition}	
\begin{proof}
It suffices to take the sequence  $u_n$ as in the proof of Proposition~\ref{C-infiniteness}.
Since  $\Psi({u}_n)=\Phi({u}_n)$ for every $n \in \N$ and 
$\| u_n\|_N  \le 1$ for $n$ sufficiently large, 
the thesis follows.
\end{proof}

The following lemma completes the proof of Theorem~\ref{sec:introduction-main-thm}. 

\begin{lemma}
\label{maximizer-radial-up-to}
Let $G$ satisfy $(G_1)$, and let $u \in \cB_\infty$ be a maximizer for $\Psi \big|_{\cB_\infty}$. Then $u = u^*$ up to a change of sign and translation.
\end{lemma}

\begin{proof}
Let $u \in \cB_\infty$ be a general maximizer of $\Psi \big|_{\cB_\infty}$. By Lemma~\ref{Riesz-rearrangement}(iii), it then follows that $u^* \in \cB_\infty^*$ is also a maximizer of $\Psi \big|_{\cB_\infty}$. Moreover, since $G(u) \in L^2(\R^N)$ by Lemma~\ref{lemma-nonuniform-finiteness-convolution} and $b_\pm(G(u),G(u))<\infty$, it also follows from Lemma~\ref{Riesz-rearrangement}(iii) that $G(u)$ equals $G(u)^*$ up to translation. Since $G$ is even and strictly increasing on $[0,\infty)$ by assumption $(G_0)$, this implies that $u$ equals $u^*$ up to sign and translation. 
\end{proof}

\section{The Euler-Lagrange equation}
\label{sec:euler-lagr-equat}
\begin{theorem}
\label{e-l-theorem-section-ball-case}
Suppose that $G \in C^1(\R)$ satisfies $(G_0)$, and suppose that $g:= G'$ satisfies
\begin{equation}
  \label{eq:f-prime-growth-condition}
g(t) \le c e^{\alpha |t|^{(N-1)/N}}\qquad \text{for $t \in \R$ with constants $c,\alpha>0$.}
\end{equation}
Suppose furthermore that $u \in \cB_1$ is  a maximizer of $\Phi \big|_{\cB_1}$. Then there exists $\theta >0$ such that $u$ satisfies the Euler-Lagrange equation in weak sense, i.e., 
	\begin{align}
          \int_{B_1} |\nabla u|^{N-2} \nabla u \nabla \phi \,dx &= \theta \, b_0(1_{B_1}G(u),1_{B_1}g(u) \phi)  \label{weak-e-l} \\ 
&= \theta  \int_{B_1} (\ln \frac{1}{|\cdot|}*\bigl(1_{B_1} G(u)\bigr)g(u)\phi\,dx \qquad \text{for all $\phi \in W^{1,N}_0(B_1).$}\nonumber
	\end{align}
\end{theorem}

Here we note that the term on the RHS of (\ref{weak-e-l}) is well-defined since $1_{B_1}G(u), 1_{B_1}g(u) \phi \in L^1_{ln}(\R^N) \cap L^2(\R^N)$ by Lemma~\ref{lemma-nonuniform-finiteness-convolution} and since these functions vanish on $\R^N \setminus B_1$. 

\begin{proof}
%$\Phi$ is differentiable in $\mathop {\mathcal{B}_1}\limits^ \circ$. 
Let $u \in \mathcal{B}_1$ is a maximer of  $\Phi \big|_{\cB_1}$. 
We first consider $\phi \in C^1_0(B_1)$ satisfying  
\begin{equation}
  \label{eq:euler-lagrange-extra-assumption}
	\int_{B_1}|\nabla u|^{N-2} \nabla u\nabla \phi \,dx<0.  
\end{equation}
We begin to prove that 
        \begin{equation}
\lim_{t \to 0^+}\frac{\Phi(u+t\phi)-\Phi(u)}{t} 
= 2b_0(1_{B_1}G(u),1_{B_1}g(u) \phi)= 2 \int_{B_1} \bigl[\ln \frac{1}{|\cdot|} * 1_{B_1} G(u)\bigr]g(u) \phi\,dx. \label{derivative-formula}
        \end{equation}
Indeed we show that there exists $t_0>0$ such that $u+t\phi \in \mathcal{B}_1$ for $0\leq t\leq t_0$.\\  For convexity we have \begin{equation*}
    \int_{B_1} |\nabla u|^{N}dx\geq \int_{B_1} |\nabla u + t\nabla \phi|^Ndx-Nt\int_{B_1} |\nabla u +t \nabla \phi|^{N-2} \nabla u \nabla \phi dx 
\end{equation*}
\begin{equation*}
    -Nt^2\int_{B_1} |\nabla u +t \nabla \phi|^{N-2} |\nabla \phi|^2dx.
\end{equation*}
We isolate the first right-hand term of the inequality 

\begin{align}\label{principale}
   &  \int_{B_1} |\nabla u|^{N}dx +Nt\int_{B_1} |\nabla u +t \nabla \phi|^{N-2} \nabla u \nabla \phi dx+ Nt^2\int_{B_1} |\nabla u +t \nabla \phi|^{N-2} | \nabla \phi|^2 dx
   \nonumber \\ &  +Nt\int_{B_1} |\nabla u|^{N-2} \nabla u \nabla \phi dx-Nt\int_{B_1} |\nabla u|^{N-2} \nabla u \nabla \phi dx
     \geq \int_{B_1} |\nabla u + t\nabla \phi|^Ndx. 
\end{align}

\bigskip
%For Lagrange Theorem applied at function %$H(t):=|\nabla u +t\nabla \phi|^{N-2}$, for %q.o.  $x\in B_1$, $\exists s_0=s_0(x) \in %(0,t)$ s.t.

%\begin{equation*}
% \frac{|\nabla u +t \nabla \phi|^{N-2} - %|\nabla u|^{N-2}}{t} = (N-2)|\nabla u+s_0\nabla %\phi|^{N-4}(\nabla u+s_0\nabla \phi) \nabla \phi
%\end{equation*}
We notice that  

	\begin{align}\label{LI}
&Nt\int_{\Omega}\bigl(|\nabla u +t\nabla \phi|^{N-2}- |\nabla u|^{N-2}|\bigr)\nabla u \nabla \phi \nonumber \\ & \leq 
Nt\int_{\Omega}||\nabla u +t\nabla \phi|^{N-2}- |\nabla u|^{N-2}||\nabla u| |\nabla \phi|\leq N(N-2)t^2 \int_{\Omega} |\nabla u +t\nabla \phi|^{N-3}|\nabla \phi |^2 |\nabla u| \nonumber
\\ &  
\leq N(N-2) t^2 2^{N-3} \int_{\Omega}
 |\nabla u|^{N-2} |\nabla \phi|^{2} + 
 N(N-2) t^{N-1} 2^{N-3} \int_{\Omega}
 |\nabla u| |\nabla \phi|^{N-1}.
\end{align}

Moreover  we have 
 \begin{align}\label{secondaria}
 &   \int_{B_1} |\nabla u + t\nabla \phi|^{N-2} |\nabla \phi|^{2} dx 	\nonumber \\ & \leq 2^{N-2}\bigg(\int_{B_1} |\nabla u|^{N-2} |\nabla \phi|^2 dx+ t^{N-2} \int_{B_1} |\nabla \phi |^{N} dx \bigg). 
  \end{align}

By \eqref{principale}, \eqref{LI} and \eqref{secondaria} we conclude
\begin{align}\label{finale}
&	 \int_{B_1} |\nabla u|^{N}dx  +
	 Nt^2
	 2^{N-2} \int_{B_1} |\nabla u|^{N-2} |\nabla \phi|^2 dx+ 
	 N	2^{N-2} t^{N} \int_{B_1} |\nabla \phi |^{N} dx \nonumber \\ & + N(N-2)t^2 2^{N-3}\int_{B_1} |\nabla u|^{N-2}|\nabla \phi|^2\, dx  
	+ N(N-2) t^{N-1 }2^{N-3}\int_{B_1}|\nabla \phi|^{N-1}|\nabla u|dx \nonumber \\ & + Nt\int_{B_1} |\nabla u|^{N-2} \nabla u \nabla \phi\, dx
	\geq \int_{B_1} |\nabla u + t\nabla \phi|^N\, dx. 
\end{align}

By (\ref{eq:euler-lagrange-extra-assumption}) and 
(\ref{finale}),  $\exists$ $t_0>0$ s.t. 
\begin{equation}\label{eq:euler-lagrange-extra-assumption-consequence}
    u+t\phi \in \mathcal{B}_1 \text{ $\ $ for $\ $ }  0\leq t\leq t_0.
\end{equation}

Let $v_t:= 1_{B_1}G(u+t \phi)$ for $0 \le t \le t_0$. Since $u + t \phi \in W^{1,N}_0(B_1)$, Lemma~\ref{lemma-nonuniform-finiteness-convolution} implies that $v_t \in L^s(\R^N)$ for any $s \in [1,\infty)$. We therefore see that, for $t \in (0,t_0)$,
	\begin{align}
	\Phi(u+t\phi)-\Phi(u) &= b_0(v_t)-
	b_0(v_0)=b_0(v_t-v_0,v_t)+ b_0(v_0,v_t-v_0) \nonumber\\
	&=t \Bigl(b_0(1_{B_1}g(u+s_t \phi)\phi,v_t)+ b_0(v_0,1_{B_1}g(u+s_t \phi)\phi)\Bigr) \label{derivative-formula-prelim}
	\end{align}
	where $s_t$ belongs to the interval $(0,t)$. Here we note that, for fixed $s \in [0,\infty)$, 
the functions $v_t$ and $w_t:= 1_{B_1}g(u+s_t \phi)$ are uniformly bounded in $L^s(\R^N)$ for $0 \le t \le t_0$. To see this, 
we remark that $
\gamma:= |u| + t_0 |\phi| \in W^{1,N}_0(B_1)$ and that 
$$
\max \{|u+s_t \phi|, |u+t \phi|\} \le \gamma  \qquad \text{in $B_1$.} 
$$
Consequently, 
$$
\max \bigl\{e^{\alpha(|u+s_t \phi|)^{N/(N-1)}}, e^{\alpha(|u+t \phi|)^{N/(N-1)}}\bigr\} \le e^{\alpha \gamma^{N/(N-1)}} \qquad \text{for $0 \le t \le t_0$,}
$$
and $e^{\alpha \gamma^{N/(N-1)}} \in L^s(B_1)$ for every $s \in [0,\infty)$ by Lemma~\ref{lemma-nonuniform-finiteness-convolution}. The claim thus follows from (\ref{eq:general-growth-condition}) and (\ref{eq:f-prime-growth-condition}).
Moreover, since $u \not = 0$ by Remark~\ref{m-F-greater-F-0} and 
$$
|u + t \phi| \to |u|\quad \text{and}\quad  |u + s_t \phi| \to |u| \in W^{1,N}_0(B_1) \qquad \text{as $t \to 0^+$,}
$$
\cite[Theorem 1.6]{Lions} implies that   
  \begin{equation}
    \label{eq:ii-l-1-convergence-e-l-eqprelimineary}
v_t \to v_0\quad \text{and}\quad  w_t \to 1_{B_1}g(u) \qquad \text{in $L^1(\R^N)$}.
  \end{equation}
Together with the uniform boundedness property we just proved, this implies, by interpolation, that 
  \begin{equation}
    \label{eq:ii-l-1-convergence-e-l-eq}
v_t \to v_0\quad \text{and}\quad  w_t \to 1_{B_1}g(u) \qquad \text{in $L^s(\R^N)$ for $1 \le s < \infty$.}
  \end{equation}
Since also $\ln^+ \frac{1}{|\cdot|} \in L^s(\R^N)$ for $1 \le s < \infty$, Young's inequality implies that 
\begin{equation}
  \label{eq:b-plus-convergence-ball-case}
b_+(w_t \phi,v_t) \to b_+(1_{B_1}g(u) \phi,v_0)\quad \text{and}\quad b_+(v_0, w_t \phi) \to 
b_+(v_0,1_{B_1}g(u) \phi)  
\end{equation}
as $t \to 0^+$.
 Moreover, using (\ref{eq:b-second-ineq}) and (\ref{eq:ii-l-1-convergence-e-l-eqprelimineary}), we easily see that 
\begin{equation}
  \label{eq:b-minus-convergence-ball-case}
b_-(w_t \phi,v_t) \to b_-(1_{B_1}g(u) \phi,v_0)\quad \text{and}\quad b_-(v_0, w_t \phi) \to 
b_-(v_0,1_{B_1}g(u) \phi).  
\end{equation}
Combining (\ref{derivative-formula-prelim}), (\ref{eq:b-plus-convergence-ball-case}) and (\ref{eq:b-minus-convergence-ball-case}), we obtain (\ref{derivative-formula}). Since $u$ is a maximizer of $\Phi$ on $\cB_1$, it follows from (\ref{derivative-formula}) and (\ref{eq:euler-lagrange-extra-assumption-consequence}) that    
\begin{equation}
  \label{eq:maximation-e-l-prelim-cond}
b_0(1_{B_1}G(u),1_{B_1}g(u) \phi) \le 0 \qquad \text{for all $\phi \in C^1_0(B_1)$ satisfying (\ref{eq:euler-lagrange-extra-assumption}). }
\end{equation}
Next we let $\phi \in C^1_0(B_1)$ satisfy
\begin{equation}
  \label{eq:euler-lagrange-extra-assumption-equality}
	\int_{B_1} |\nabla u|^{N-2}\nabla u \nabla \phi \,dx=0.  
\end{equation}
Approximating $\phi$ and $-\phi$ with functions satisfying (\ref{eq:euler-lagrange-extra-assumption}), we  deduce from (\ref{eq:maximation-e-l-prelim-cond}) that 
\begin{equation}
  \label{eq:maximation-e-l-prelim-cond-equality}
0= b_0(1_{B_1}G(u),1_{B_1}g(u) \phi) = \int_{B_1} (\ln \frac{1}{|\cdot|}*\bigl(1_{B_1} G(u)\bigr)g(u)\phi\,dx 
\end{equation}
for all $\phi \in C^1_0(B_1)$ satisfying (\ref{eq:euler-lagrange-extra-assumption-equality}). Since $g(u) \in L^s(B_1)$ for $s \in [1,\infty)$, it follows from Lemma~\ref{lemma-nonuniform-finiteness-convolution}(i)
that
$$
(\ln \frac{1}{|\cdot|}*\bigl(1_{B_1} G(u)\bigr)g(u) \qquad \in L^s(B_1) \qquad \text{for $s \in [1,\infty)$.} 
$$
Consequently, the RHS of (\ref{eq:maximation-e-l-prelim-cond-equality}) defines a continuous linear form on $W^{1,N}_0(B_1)$. By density, it thus follows that (\ref{eq:maximation-e-l-prelim-cond-equality}) holds for all $\phi \in W^{1,N}_0(B_1)$  satisfying (\ref{eq:euler-lagrange-extra-assumption-equality}). Hence there exists $\Theta \in \R$ with the property that 
\begin{equation}
  \label{eq:capital-Theta-property}
\Theta	\int_{B_1} |\nabla u|^{N-2}\nabla u \nabla \phi \,dx =\int_{B_1} (\ln \frac{1}{|\cdot|}*\bigl(1_{B_1} G(u)\bigr)g(u)\phi\,dx\qquad \text{for all $\phi \in W^{1,N}_0(B_1).$}
\end{equation}
Next we claim that $\Theta>0$. For this we recall that $u \in \cB_1^*$ by Lemma~\ref{maximizer-strictly-positive}. In particular, $u$ is a radial function. Applying (\ref{eq:capital-Theta-property}) to $\phi=u$ and using Corollary~\ref{cor-newtons-theorem}, we deduce that 
\begin{equation}
  \label{eq:capital-Theta-property-applied-u}
\Theta |\nabla u|_N^N = \int_{B_1} (\ln \frac{1}{|\cdot|}*\bigl(1_{B_1} G(u)\bigr)g(u)u\,dx = 2\omega_{N-1}^2 \int_0^1 r^{N-1} h_u g(u)u\,dr 
\end{equation}
with the function 
$$
r \mapsto h_u(r)= \ln \frac{1}{r} \int_{0}^r 
\rho^{N-1} G(u(\rho)) d\rho +  \int_{r}^1  \rho^{N-1} \bigl(\ln \frac{1}{\rho}\bigr)G(u(\rho)) d\rho.  
$$
By Lemma~\ref{maximizer-strictly-positive}, we have $u>0$ on $(0,1)$ and therefore also $   h_u>0$ on $(0,1)$. 
Moreover, we have $g >0$ a.e. on $(0,\infty)$ by assumption $(G_0)$, which implies in particular that $g(u) \ge 0$. Consequently, (\ref{eq:capital-Theta-property-applied-u}) implies that $\Theta \ge 0$. Moreover, if we suppose by contradiction that $\Theta=0$, then we have 
$$
\int_0^1 r^{N-1}h_u g(u)u\,dr = 0
$$
 and therefore $g(u) \equiv 0$. Since $u$ is continuous on $(0,1]$ with $u(1)=0$ and $g>0$ a.e. on $(0,\infty)$, it then follows that $u \equiv 0$, a contradiction. We thus obtain $\Theta>0$, as claimed. We conclude that (\ref{weak-e-l}) holds with $\theta = \frac{1}{\Theta}>0$. The proof is finished.
\end{proof}

Next we derive the following result on the entire space.

\begin{theorem}
\label{e-l-theorem-section-B-infty-case}
Suppose that $G \in C^1(\R)$ satisfies $(G_0)$ and $G(0)=0$. Moreover, suppose that $g:= G'$ satisfies (\ref{eq:f-prime-growth-condition}). 
	If $u \in \cB_\infty$ is  a maximizer of $\Psi\big|_{\cB_\infty}$, then there exists $\theta >0$ such that $u$ satisfies the Euler-Lagrange equation in weak sense, i.e., 
	\begin{equation}
	\label{weak-e-l-r-2-case}
	\int_{\R^N}\bigl(|\nabla u|^{N-2} \nabla u\nabla \phi + |u|^{N-2}u \phi \bigr) \,dx = \theta b_0(G(u),g(u)\phi) = \theta 
	\int_{\R^N} \bigl(\ln \frac{1}{|\cdot|}* G(u)\bigr)g(u)\phi\,dx 
	\end{equation}
for all $\phi \in W^{1,N}(\R^N)$ with bounded support.
\end{theorem}

Here we note that $u= u^*$ up to translation by Lemma~\ref{maximizer-radial-up-to}, so $G(u) \in L^2(\R^N)$ and $g(u) \phi \in L^2(\R^N)$ by Lemma~\ref{lemma-nonuniform-finiteness-convolution}. Since $\phi$ has compact support, we also see that $g(u) \phi 
\in L^1_{ln}(\R^N)$. Moreover, $G(u) \in L^1_{ln}(\R^N)$ by Corollary~\ref{converse-inequality-remark} since $\Psi_-(u)= b_-(G(u),G(u))< \infty$. Hence the term on the RHS of (\ref{weak-e-l-r-2-case}) is well-defined.

\begin{proof}
First we recall that $u=u^*$ up to sign and translation, so we may assume that $u \in \cB_\infty^*$ in the following.
Let first $\phi \in C^1_0(\R^N)$, $\phi$ radial, and assume that 
\begin{equation}
  \label{eq:euler-lagrange-extra-assumption-r-2}
	\int_{B_1}\bigl(|\nabla u|^{N-2} \nabla u\nabla \phi + |u|^{N-2}u \phi \bigr) \,dx<0.  
\end{equation}
As in the proof of Theorem \ref{e-l-theorem-section-ball-case} exists $t_0>0$  with 
\begin{equation}
  \label{eq:euler-lagrange-extra-assumption-consequence-r-2}
u + t \phi  \in \cB_\infty \qquad \text{for $0 \le t \le t_0$.}  
\end{equation}
We will recognize  that 
        \begin{equation}
\lim_{t \to 0^+}\frac{\Psi(u+t\phi)-\Psi(u)}{t} = 2 b_0(G(u),g(u)\phi) = 2\int_{\R^N} \bigl(\ln \frac{1}{|\cdot|}* G(u)\bigr) g(u)\phi\,dx.   \label{derivative-formula-B-infty-1}
        \end{equation}
For this, we first note that $v_t:= G(u+t \phi) \in L^s_{loc}(\R^N)\cap L^2(\R^N)$ for $1 \le s < \infty$ by Lemma~\ref{lemma-nonuniform-finiteness-convolution}. 
Moreover, $v_t \in L^1_{ln}(\R^N)$ since $G(u) \in L^1_{ln}(\R^N)$ and $\phi$ has bounded support. 
 We therefore see that, for $t \in (0,t_0)$,
	\begin{align}
	\Phi(u+t\phi)-\Phi(u) &= b_0(v_t)-
	b_0(v_0)=2 b_0(v_t-v_0,v_0) + b_0(v_t-v_0) \nonumber\\
	&=t \Bigl(b_0\bigl(g(u+s_t \phi)\phi,v_0\bigr)+ t b_0\bigl(g(u+s_t \phi)\phi \bigr) \Bigr) \label{derivative-formula-prelim-r-2-case}
	\end{align}
	where $s_t$ belongs to the interval $(0,t)$. Here we note that, for fixed $s \in [0,\infty)$, 
the functions $v_t$ and $w_t:= 1_{B_1}g(u+s_t \phi)$ are uniformly bounded in $L^s_{loc}(\R^N)$ for $0 \le t \le t_0$. 
Indeed, set $\gamma:=|u|+t_0 |\phi| \in W^{1,N}(\R^N)$, we have 
$$
\max \bigl\{e^{\alpha(|u+s_t \phi|)^{N/(N-1)}}, e^{\alpha(|u+t \phi|)^{N/(N-1)}}\bigr\} \le e^{\alpha \gamma^{N/(N-1)}} \qquad \text{for $0 \le t \le t_0$,}
$$
and
$e^{\alpha \gamma^{N/(N-1)}} \in L^s_{loc}(B_1)$ for every $s \in [0,\infty)$ by Lemma~\ref{lemma-nonuniform-finiteness-convolution}. The claim thus follows from (\ref{eq:general-growth-condition}) and (\ref{eq:f-prime-growth-condition}). Moreover, since $u \not = 0$ by Remark~\ref{m-F-greater-F-0} and 
$$
|u + t \phi| \to |u|\quad \text{and}\quad  |u + s_t \phi| \to |u| \in W^{1,N}(\R^N) \qquad \text{as $t \to 0^+$,}
$$
\cite[Theorem 1.6]{Lions} implies that   
  \begin{equation}
    \label{eq:ii-l-1-convergence-e-l-eqprelimineary-r-2}
v_t \to v_0\quad \text{and}\quad  w_t \to g(u) \qquad \text{in $L^1_{loc}(\R^N)$}.
  \end{equation}
Together with the uniform boundedness property we just proved, this implies, by interpolation, that 
  \begin{equation}
    \label{eq:ii-l-1-convergence-e-l-eq-r-2-case}
w_t  \to g(u) \qquad \text{in $L^s_{loc}(\R^N)$ for $1 \le s < \infty$.}
  \end{equation}
Consequently, since $\phi$ has compact suppport, 
  \begin{equation}
    \label{eq:ii-l-1-convergence-e-l-eq-r-2-case-1}
w_t \phi \to g(u)\phi \qquad \text{in $L^s(\R^N)$ for $1 \le s < \infty$ and in $L^1_{ln}(\R^N)$.}
  \end{equation}
Since, as noted earlier, $v_0 \in L^2(\R^N)$ and $\ln^+ \frac{1}{|\cdot|} \in L^s(\R^N)$ for $1 \le s < \infty$, Young's inequality implies that 
\begin{equation}
  \label{eq:b-plus-convergence-r-2-case}
b_+(w_t \phi,v_0) \to b_+(g(u) \phi,v_0)\quad \text{and}\quad b_+(w_t \phi) \to b_+(g(u)\phi)
\end{equation}
as $t \to 0^+$. Moreover, since also $v_0 \in L^1_{ln}(\R^N)$, we deduce from (\ref{eq:b-second-ineq}) and (\ref{eq:ii-l-1-convergence-e-l-eq-r-2-case-1}) that 
\begin{equation}
  \label{eq:b-minus-convergence-r-2-case}
b_-(w_t \phi,v_0) \to b_-(g(u) \phi,v_0)\quad \text{and}\quad b_-(w_t \phi) \to b_-(g(u)\phi). 
\end{equation}
Combining (\ref{derivative-formula-prelim-r-2-case}), (\ref{eq:b-plus-convergence-r-2-case}) and (\ref{eq:b-minus-convergence-r-2-case}), we infer (\ref{derivative-formula-B-infty-1}). 
Now since $u$ is a maximizer of $\Phi$ on $\cB_\infty$, it follows that the RHS in (\ref{derivative-formula-B-infty-1}) is less than or equal to zero 
for all $\phi \in C^1_0(\R^N)$ satisfying (\ref{eq:euler-lagrange-extra-assumption-r-2}). Precisely as in the proof of Theorem~\ref{e-l-theorem-section-ball-case}, we then deduce that 
\begin{equation}
  \label{eq:maximation-e-l-prelim-cond-equality-B-infty}
b_0(G(u),g(u)\phi) = 0  
\end{equation}
for all functions $\phi \in C^1_0(\R^N)$ satisfying 
\begin{equation}
  \label{eq:euler-lagrange-extra-assumption-equality-B-infty}
	\int_{\R^N}\bigl(|\nabla u|^{N-2} \nabla u\nabla \phi + |u|^{N-2}u \phi \bigr) \,dx=0.  
\end{equation}
Moreover, by (\ref{eq:general-growth-condition}), (\ref{eq:f-prime-growth-condition}) and Lemma~\ref{lemma-nonuniform-finiteness-convolution}, we have 
$$
\bigl(\ln \frac{1}{|\cdot|}* G(u)\bigr)g(u) \qquad \in L^s_{loc}(\R^N) \qquad \text{for $s \in [1,\infty)$.} 
$$
Consequently, for arbitrary fixed $R>0$, the RHS of (\ref{derivative-formula-B-infty-1}) defines a continuous linear form on $W^{1,N}_0(B_R)$. By density, it thus follows that (\ref{eq:euler-lagrange-extra-assumption-equality-B-infty}) holds for all $\phi \in W^{1,N}_0(B_R)$  satisfying (\ref{eq:euler-lagrange-extra-assumption-equality}). 
Hence there exists $\Theta \in \R$ with the property that 
\begin{equation}
  \label{eq:capital-Theta-property-B-infty}
\Theta	\int_{\R^N}\bigl(|\nabla u|^{N-2} \nabla u\nabla \phi + |u|^{N-2}u \phi \bigr) \,dx =\int_{\R^N} 
\bigl(\ln \frac{1}{|\cdot|}* G(u)\bigr)g(u)\phi\,dx \qquad \text{for all $\phi \in W^{1,N}_0(B_R).$}
\end{equation}
Since $W^{1,N}_0(B_R) \subset W^{1,N}_0(B_{R'})$ if $R<R'$, it follows that $\Theta$ does not depend on $R$ and therefore (\ref{eq:capital-Theta-property-B-infty}) holds for all $\phi \in W^{1,N}(\R^N)$ with bounded support.  

Next we claim that $\Theta \not = 0$, and we suppose by contradiction that $\Theta = 0$. Then we have 
$$
(\ln \frac{1}{|\cdot|}*\bigl(1_{B_1} G(u)\bigr)g(u) = 0 \qquad \text{a.e. in $\R^N$.}
$$
By Lemma~\ref{sec:newtons-theorem} and since $u$ is a radial function, this implies that 
\begin{equation}
  \label{eq:equality-contradiction-condition-B-infty}
h_u(r) g(u(r)) = 0 \quad \text{for a.e. $r \in (0,\infty)$}
\end{equation}
with 
$$
h_u(r)= \ln \frac{1}{r} \int_{0}^r 
\rho^{N-1} G(u(\rho)) d\rho +  \int_{r}^\infty  \rho^{N-1}  \bigl(\ln \frac{1}{\rho}\bigr)G(u(\rho)) d\rho. 
$$
Since $u \not \equiv 0$ by Remark~\ref{m-F-greater-F-0}, we have $G(u) \not \equiv 0$ and therefore $h_u(r)<0$ for $r>1$. Hence (\ref{eq:equality-contradiction-condition-B-infty}) implies that $g(u) \equiv 0$ on $(1,\infty)$. Noting again that $g(s) \not = 0$ a.e. on $\R$ by assumption $(G_0)$, we deduce again that $u$ is constant on $(1,\infty)$. Since $u \in W^{1,N}_0(\R^N)$, this implies that $u \equiv 0$ on $(1,\infty)$ and therefore $G(u) \equiv G(0)= 0$ on $(1,\infty)$. As a consequence, $G(u) \not \equiv 0$ on $(0,1)$ and therefore  
$$
h_u(r)= \ln \frac{1}{r} \int_{0}^r 
\rho^{N-1} G(u(\rho)) d\rho +  \int_{r}^1  \rho^{N-1}  \bigl(\ln \frac{1}{\rho}\bigr)G(u(\rho)) d\rho > 0 \qquad \text{for $r \in (0,1)$.} 
$$
Then (\ref{eq:equality-contradiction-condition-B-infty}) implies that $g(u) \equiv 0$ on $(0,1)$. As above, we then deduce that $u$ is constant on $(0,1)$. Since we already know that $u$ is continuous on $(0,\infty)$ and satisfies $u \equiv 0$ on $[1,\infty)$, we conclude that $u \equiv 0$ on $(0,\infty)$, contradicting Remark~\ref{m-F-greater-F-0}.  Hence $\Theta \not = 0$.
 
Next we show that $\Theta>0$. For this, we first note that the radial function  
$$
r \mapsto f_u(r)=  \Bigl( \ln \frac{1}{|\cdot|} \ast G(u) \Bigr)(r)=-  \omega_{N-1} \Bigl(\ln r \int_{0}^r \rho^{N-1}  G(u(\rho))d\rho + \int_{r}^\infty \rho^{N-1} \ln \rho G(u(\rho)) d\rho \Bigr)
$$
is decreasing in $r$. Indeed we have  
\begin{align*}
f'_u(r) & =  - \omega_{N-1} \Bigl( 
\frac{1}{r}  \int_{0}^r  G(u(\rho)) \rho^{N-1} d \rho + r^{N-1} \ln r G(u(r)) - r^{N-1} \ln r G(u(r))\Bigr) \\ & =
-\frac{\omega_{N-1}}{r}  \int_0^{r} G(u(\rho)) \rho^{N-1} d \rho \le 0 \qquad \text{for $r \in (0,1)$.}
\end{align*}
If $f_u(r) \leq 0$ for all $r \in (0,1)$, then 
$$
\Psi(u)= \int_{\R^N} G(u) f_u dx \leq 0 \le \Psi(0)
$$
which is wrong for maximizers by Remark~\ref{m-F-greater-F-0}. Hence there exists $r_0 \in (0,1)$ with $h_u(r)>0$ for $r \in (0,r_0)$. Making $r_0$ smaller if necessary, we may also assume that $u(r_0) >0$. Applying (\ref{eq:capital-Theta-property-B-infty}) to the function $\phi= [u-u(r_0)]_+ \in W^{1,N}_0(B_1)$, we see that
\begin{equation}
  \label{eq:capital-Theta-property-special-phi}
\Theta	\int_{B_{r_0}}\bigl(|\nabla u|^N + |u|^{N-2} u[u-u(r_0)] \bigr) \,dx
 = \int_{B_{r_0}}g(u)f_u [u-u(r_0)] \,dx.
\end{equation}
We distinguish two cases.\\
{\bf Case 1:} $u \equiv u_{r_0}$ on $B_{r_0}$. In this case it follows from (\ref{eq:capital-Theta-property-B-infty}) that 
$$
\Theta u^{N-1}(r_0) \int_{B_{r_0}} \phi \,dx = g(u(r_0)) \int_{B_{r_0}}f_u \phi \,dx \qquad \text{for all $\phi \in W^{1,N}_0(B_{r_0})$}
$$
and therefore 
\begin{equation}
  \label{eq:Case-1-Theta-property}
\Theta u^{N-1}(r_0) = g(u(r_0)) f_u  \qquad \text{a.e. on $B_{r_0}$.}
\end{equation}
Since $\Theta \not = 0$, $u(r_0) >0$, $g(u(r_0))\ge 0$ and $f_u>0$ on $B_{r_0}$, we deduce that $\Theta>0$ in this case.\\ 
{\bf Case 2:} $u \not \equiv u_{r_0}$ on $B_{r_0}$. In this case it follows from (\ref{eq:capital-Theta-property-special-phi}) that 
$$
\Theta  = \Bigl(\int_{B_{r_0}}\bigl(|\nabla u|^N + |u|^{N-2} u [u-u(r_0)] \bigr) \,dx\Bigr)^{-1} \int_{B_{r_0}}g(u)f_u [u-u(r_0)] \,dx \ge 0.
$$
Since we have already proved that $\Theta \not = 0$, we conclude again that $\Theta>0$.\\ 
So in either case we have $\Theta>0$, and then (\ref{weak-e-l-r-2-case}) holds with $\theta = \frac{1}{\Theta}>0$. The proof is finished.
\end{proof}

 \bigskip \bigskip
\begin{remark}\label{syst} The formal equivalence between the Choquard type equation and the higher order fractional system has been discussed in \cite{BucurCassani}.	

The key step was to consider the limiting case of the fractional Laplacian $(-\Delta)^{\frac{N}{2}}$ when $\frac{N}{2}$ is odd and even. More precisely, 
let us consider the fractional Poisson's equation 
    \begin{equation} \label{fractional-laplacian}
        (-\Delta)^s u=f \textit{$\ $ $\ $ in $\ $} \R^N \textit{$\ $ $\ $ with $\ $} 0<s\leq \frac{N}{2}.
    \end{equation}
Set for $s>0$, \begin{equation*}
    L_s(\R^N):=\bigg\{ u \in L^1_{loc}(\R^N)| \int_{\R^N} \frac{|u(x)|}{1+|x|^{N+2s}} dx< +\infty \bigg\},
\end{equation*}the operator $(-\Delta)^s u$ is defined for all $u \in L_s(\R^N)$ via duality, as
\begin{equation*}
    \langle (-\Delta)^s u, \phi \rangle =\int_{\R^N} u(-\Delta)^s \phi dx, \textit{$\ $ $\ $} \forall \phi \in \S(\R^N),
\end{equation*}where \begin{equation*}
    (-\Delta)^s \phi =\mathcal{F}^{-1} \bigg( |\xi|^{2s} \mathcal{F}\phi (\xi) \bigg), \textit{$\ $ $\ $ $\ $} \forall \phi \in \S(R^N)
\end{equation*} denoting by $\mathcal{F}$ the Fourier transform and where $\S(\R^N)$ denotes the Schwartz space of rapidly decreasing functions.
    
    If $s\in (0,1)$, the setting and the representation formulas for solutions of this equation are settled in literature, among which recent studies carried out in
    \cite{Abatangelo, Bucur}. The cases $s>1$ has been considered in very recent papers \cite{AbatJS, Brasco, Stinga}, whereas a general approach based on the notion of distributional solution dates back to classical works \cite{Landkof,Samko, Stein}, see also \cite{Galdi}.

\begin{definition}
    Given  $f \in \S'(\R^N)$, we say that $u \in L_{\frac{N}{2}}(\R^N)$ is a distributional solution of \eqref{fractional-laplacian}
    \begin{equation*}
        \int_{\R^N} u (-\Delta)^{\frac{N}{2}}\phi dx=\langle f, \phi\rangle 
    \end{equation*}
    for all $\phi \in \S(\R^N)$.
\end{definition}
It is well know that, if $s < \frac{N}{2}$, the distributional solution of \eqref{fractional-laplacian} is given by \begin{equation*}
    u(x)=\mathcal{F}^{-1}\bigg(\frac{1}{|\xi|^{2s}}\mathcal{F}f(\xi)\bigg)(x)   
\end{equation*}
which is realized also by convolution with the Newtonian potential
\begin{equation*}
    u(x)=(I_{2s} * f)(x) \textit{ , $\ $ $\ $ where $\ $ $\ $} I_{2s}(x)=\frac{1}{\gamma_{N,s}} |x|^{2s-N}
\end{equation*}
for some $\gamma_{N,s}>0$, (see \cite[Chapter 5]{Samko}). Note that the Newtonian potential, in the case $2s=N$, is given by \begin{equation*}
    I_N(x)=\frac{1}{\gamma_N}\log \frac{1}{|x|}= \mathcal{F} |\xi|^{-N}(x).
\end{equation*}
When $2s=N$ it is not possible to define the solution to \eqref{fractional-laplacian} by Fourier trasform in $\S(\R^N)$ in general, due to the singularity of $|\xi|^{-N}$.\\ Notice that when $N$ is even, $(-\Delta)^{\frac{N}{2}}$ is a integer operator, so its fundamental solution in $\R^N$ is known, see e.g. \cite[Proposition 22]{Martinazzi}.\\ When $N $ is odd, the fractional case $2s=N$ is given in \cite{Hyder}, in which the $\frac{N}{2}$-Laplacian can be seen as the composition of the $\frac{1}{2}$-Laplacian with the Laplacian of integer order $\frac{N-1}{2}$. 
\\ \noindent
\end{remark} 
\bigskip

\noindent
{\bf Acknowledgments.}
The authors wish to thank the anonymous referee for the remarkable comments which helped us to improve the paper.
The second author wishes to thank Prof. Tobias Weth for fruitful discussions on the topic.
The authors are supported by PRIN PNRR  P2022YFAJH {\sl \lq\lq Linear and Nonlinear PDEs: New directions and applications''} (CUP H53D23008950001), and partially supported by INdAM-GNAMPA.
The second author thanks acknowledge financial
support from PNRR MUR project PE0000023 NQSTI - National Quantum Science and Technology
Institute (CUP H93C22000670006).


\bigskip
 
 
 
 

\begin{bibdiv}
		\begin{biblist}
			\bib{Abatangelo}{article}{
				AUTHOR={N. Abatangelo},
				TITLE={Very large solutions for the fractional Laplacian: towards a fractional Keller-Osserman condition},
				JOURNAL = {Adv. Nonlinear Anal.},
				FJOURNAL = {},
				VOLUME = {6},
				YEAR = {2017},
				NUMBER = {},
				PAGES = {383-405},}
		
			
				
			\bib{AbatJS}{article}{
				AUTHOR={Abatangelo, N.},
				AUTHOR={Jarohs, S.},
				AUTHOR={Salda$\mathrm{\tilde{n}}$a, A.},
				TITLE={Integral representation of solutions to higher-order fractional Dirichlet problems on balls},
				JOURNAL = {Commun. Contemp. Math.},
				FJOURNAL = {},
				VOLUME = {20},
				YEAR = {2018},
				NUMBER = {},
				PAGES = {36pp},}
			
			\bib{AY}{article}{
				AUTHOR = {Adimurthi, S., and Y. Yang ,}
				TITLE = {An interpolation of Hardy inequality and Trudinger-Moser inequality in $\R^N$ and its
					applications},
				JOURNAL = {Int. Math. Res. Not.},
				FJOURNAL = {IMRN},
				VOLUME = {13},
				YEAR = {2010},
				NUMBER = {},
				PAGES = {2394--2426},}
			
	
			
	
			
			\bib{beckner}{article}{
				AUTHOR = {Beckner, W.},
				TITLE = {Sharp Sobolev inequalities on the sphere and the Moser-Trudinger inequality},
				JOURNAL = {Ann. of Math. (2)},
				FJOURNAL = {Annals of Mathematics. Second Series},
				VOLUME = {138},
				YEAR = {1993},
				NUMBER = {1},
				PAGES = {213--242},
				ISSN = {0003-486X},
				MRCLASS = {58G30 (46E35 53C21 58G26)},
				MRNUMBER = {1230930},
				MRREVIEWER = {Paul C. Yang},
				%DOI = {10.2307/2946638},
				%URL = %{https://doi.org/10.2307/294%6638},
			}	
			
			\bib{BL}{article}{
				AUTHOR = {Berestycki, H., and P.L. Lions,}
				TITLE = {Nonlinear Scalar field equations, I. Existence of ground state},
				JOURNAL = {Arch. Rational Mech. Anal.},
				FJOURNAL = {Arch. Rational Mech. Anal.},
				VOLUME = {82},
				YEAR = {1983},
				NUMBER = {},
				PAGES = {313--346},
			}
			
			
			
				
			
			\bib{BCS}{article}{
				AUTHOR = {Bonheure, D.},
				AUTHOR = {Cingolani, S.},
				AUTHOR = {Secchi, S.},
				TITLE = {Concentration phenomena for the Schr\"odinger-Poisson system in $\R^2$},
				JOURNAL = {Discrete Contin. Dyn. Syst. Ser. S},
				FJOURNAL = {DCDS-S},
				VOLUME = {14},
				YEAR = {2021},
				NUMBER = {5},
				PAGES = {1631--1648},
			}
			
			
			
			
			
			
			\bib{BVS}{article}{
				AUTHOR = {Bonheure, D.},
				AUTHOR = {Cingolani, S.},
				AUTHOR = {Van Schaftingen, J.},
				TITLE = {The logarithmic Choquard equation: sharp asymptotics and nondegeneracy of the groundstate},
				JOURNAL = {J. Funct. Anal.},
				FJOURNAL = {J. Funct. Anal.},
				VOLUME = {272},
				YEAR = {2017},
				NUMBER = {},
				PAGES = {5255--5281},
			}
			
			\bib{Brasco}{article}{
				AUTHOR={Brasco, L.},
				AUTHOR={Gomez-Castro, D.},
				AUTHOR={Vazquez, J.L.},
				TITLE={Characterisation of homogeneous fractional Sobolev spaces}
				JOURNAL = {Calc. Var. Partial Differ. Equ.},
				FJOURNAL = {Calc. Var. Partial Differ. Equ.},
				VOLUME = {60},
				YEAR = {2021},
				NUMBER = {},
				PAGES = {40pp},
					}
			
			
			\bib{Bucur}{article}{
				AUTHOR={Bucur, C.},
				TITLE={Some observations on the Green function for the ball in the fractional Laplace framework}
				JOURNAL = {Commun. Pure Appl. Anal.},
				FJOURNAL = {},
				VOLUME = {15},
				YEAR = {2016},
				NUMBER = {},
				PAGES = {657--699},
			}
			
			\bib{BucurCassani}{article}{
				AUTHOR = {Bucur, C.}
				AUTHOR = {Cassani, D.}
				AUTHOR = {Tarsi, C.}
				TITLE =  {Quasilinear logarithmic Choquard equations with exponential growth in $\R^N$},
				JOURNAL = {Journal of Differential Equations},
				FJOURNAL = {},
				VOLUME = {328},
				YEAR = {2022},
				NUMBER = {},
				PAGES = {261--294},
			}
			
			
			
			
			
			
			
			
			
			
			\bib{CLMP}{article}{
				AUTHOR = {Caglioti, E.},
				AUTHOR = {Lions, P.L.},
				AUTHOR = {Marchioro, C.},
				AUTHOR = {Pulvirenti, M.}
				TITLE = {A special class of stationary flows for two-dimensional Euler equations: A statistical mechanics description. Part II},
				JOURNAL = {Comm. Math. Phys.},
				FJOURNAL = {C.M.P.},
				VOLUME = {174},
				YEAR = {1995},
				NUMBER = {},
				PAGES = {229--260},
			}
			
			
			\bib{cao}{article}{
				AUTHOR = {Cao, D.M.}
				TITLE = {Nontrivial solution of semilinear elliptic equation with critical exponent in $\R^2$},
				JOURNAL = {Comm. P.D.E.},
				FJOURNAL = {Communications Partial Differential Equations},
				VOLUME = {17},
				YEAR = {1992},
				NUMBER = {},
				PAGES = {407--435},
			}
			
			
			
			\bib{carlesonchang}{article}{
				AUTHOR = {Carleson, L., and A. Chang,}
				TITLE = {On the existence of an extremal function for an inequality of J. Moser},
				JOURNAL = {Bull. Sci. Math.},
				FJOURNAL = {},
				VOLUME = {110},
				YEAR = {1986},
				NUMBER = {},
				PAGES = {113--127},
			}
			
			
			
			
			
			\bib{CassaniLiuRomani}{article}{
				AUTHOR = {Cassani,D.}
				AUTHOR = {Liu, Z.}
				AUTHOR = {Romani, G.}
				TITLE =  {Nonlocal planar Schr\"odinger-Poisson systems in the fractional Sobolev limiting case},
				JOURNAL = {Journal of Differential Equations},
				FJOURNAL = {},
				VOLUME = {384},
				YEAR = {2024},
				NUMBER = {},
				PAGES = {214--269},
			}
			
			\bib{CassaniLiuRomani1}{article}{
				AUTHOR = {Cassani,D.}
				AUTHOR = {Liu, Z.}
				AUTHOR = {Romani, G.}
				TITLE =  {Nonlocal  Schr\"odinger-Poisson systems in $\R^N$},
				JOURNAL = {preprint Arxiv:2311.13424},
				FJOURNAL = {},
				VOLUME = {},
				YEAR = {},
				NUMBER = {},
				PAGES = {},
			}
			
			
			
			

			
			
			\bib{CiJe}{article}{
				AUTHOR = {Cingolani, S., and L. Jeanjean,}
				TITLE = {Stationary solutions with prescribed $L^2$-norm for the planar Schr\"odinger - Poisson system},
				JOURNAL = {SIAM J. Math. Anal.},
				FJOURNAL = {SIAM J. Math. Anal.},
				VOLUME = {51},
				YEAR = {2019},
				NUMBER = {},
				PAGES = {3533--3568},
			}
			
			
			\bib{CiWe}{article}{
				AUTHOR = {Cingolani, S., and T. Weth,}
				TITLE = {On the planar Schr\"odinger-Poisson system},
				JOURNAL = {Ann. Inst. H. Poincar\'e Anal. Non Lin\'eaire},
				FJOURNAL = {Ann. Inst. H. Poincar\'e Anal. Non Lin\'eaire},
				VOLUME = {33},
				YEAR = {2016},
				NUMBER = {},
				PAGES = {169--197},
			}
			
			\bib{CiWe2}{article}{
				AUTHOR = {Cingolani, S., and T. Weth,}
				TITLE = {Trudinger-Moser type inequality with logarithmic convolution potentials},
				JOURNAL = {Journal of London Mathematical Society},
				FJOURNAL = {Journal of London Mathematical Society},
				VOLUME = {105},
				YEAR = {2022},
				NUMBER = {3},
				PAGES = {1897
			--1935},
			}
			
			
			
			\bib{CiWeYu}{article}{
				AUTHOR = {Cingolani, S.}
				AUTHOR = {Weth, T.},
				AUTHOR = {Yu, M.}
				TITLE = {Extremal functions for the critical
					Trudinger-Moser inequality with logarithmic kernels},
				JOURNAL = {ESAIM COCV},
				FJOURNAL ={},
				VOLUME = {30},
				YEAR = {2024},
				NUMBER = {75},
				PAGES = {pp.25},
			}
			
			
			
			
			
			
			
			
			
			
			
			
			\bib{Dolbeault-Perthame}{article}{
				AUTHOR = {Dolbeault, J., and B. Perthame,}
				TITLE = {Optimal critical mass in the two dimensional Keller-Segel model in $\R^2$}
				JOURNAL = {C. R. Acad. Sci. Paris, Ser. I}
				FJOURNAL = {C. R. Acad. Sci. Paris, Ser. I}
				VOLUME = {339},
				YEAR = {2004},
				NUMBER = {},
				PAGES = {611--616},
			}
			
			
			
			
			
			
			
			
			
			\bib{DuWe}{article}{
				AUTHOR = {Du, M., and T. Weth,}
				TITLE = {Ground states and high energy solutions of the planar Schr\"odinger-Poisson system,},
				JOURNAL = {Nonlinearity},
				FJOURNAL = {Nonlinearity},
				VOLUME = {30},
				YEAR = {2017},
				NUMBER = {},
				PAGES = {3492--3515},
			}
			
			
			
			\bib{Doo-Marcos}{article}
			{AUTHOR={do \'O, J.M.} 
				TITLE={N-Laplacian equations in $\R^N$ with critical growth}
				JOURNAL = {Abstr. Appl. Anal.},
				FJOURNAL = {},
				VOLUME = {2},
				YEAR = {1997},
				NUMBER = {301-315},
				PAGES = {},
			}
			
			
			
			
			%\bib{Lieb1983}}{article}{
			%	AUTHOR = {Lieb, E. H.},
			%	TITLE = {Sharp constants in the Hardy-Littlewood-Sobolev and related inequalities},
			%	JOURNAL = {The Annals of Mathematics},
			%	FJOURNAL = {The Annals of Mathematics},
			%	VOLUME = {118},
			%	YEAR = {1983},
			%	NUMBER = {},
			%	PAGES = {349--374},
			%}
		\bib{Galdi}{article}{
			AUTHOR = {Galdi, G.P.}
			TITLE = {An Introduction to the Mathematical Theory of the Navier-Stokes Equations. Steady-State Problems},
			JOURNAL={(second edition), Springer Monographs in Mathematics, Springer, New York},
			VOLUME = {2},
			YEAR = {2011},
		}
		
		\bib{Hyder}{article}{
			AUTHOR = {Hyder, A.}
			TITLE = {Structure of conformal metrics on $\R^N$ with constant Q-curvature},
			JOURNAL = {Differ. Integral Equ.},
			FJOURNAL = {},
			VOLUME = {32},
			YEAR = {2019},
			NUMBER = {},
			PAGES = {423--454},
		}
		
		
		\bib{Landkof}{article}{
			AUTHOR = {Landkof, N.S.}
			TITLE = {Foundations of Modern Potential Theory },
			JOURNAL={Die Grundlehren der Mathematischen Wissenschaften Springer-Verlag, New York-Heidelberg},
			VOLUME = {180},
			YEAR = {1972},
		}
		
		\bib{LiebLoss}{book}{
			AUTHOR = {Lieb, E. H., and M. Loss},
			TITLE = {Analysis},
			BOOK = {AMS},
			VOLUME = {14},
			YEAR = {2001},
		}
		
		
		\bib{Lions}{article}{
			AUTHOR = {Lions, P.L.},
			TITLE = {The concentration-compactness principle in the calculus of variations. The limit case, part 1},
			JOURNAL = {Riv. Mat. Iberoamericana},
			FJOURNAL = {Riv. Mat. Iberoamericana},
			VOLUME = {1},
			YEAR = {1985},
			PAGES = {145--201},
		}
		
		
		\bib{LuZhu}{article}{
			AUTHOR={Lu, G., and M. Zhu},
			TITLE={A sharp Trudinger-Moser type inequality involving $L^N$ norm in the entire space $\R^N$},
			JOURNAL={ Journal of Differential Equations},
			VOLUME = {267},
			YEAR = {2019},
			NUMBER = {},
			PAGES={3046--3082},
		}
		
		
		
		\bib{Martinazzi}{article}{
			AUTHOR = {Martinazzi, L.},
			TITLE = {Classification of solutions to the higher order Liouville's equation on $\R^{2m}$},
			JOURNAL = {Math. Z.},
			FJOURNAL = {},
			VOLUME = {263},
			YEAR = {2009},
			NUMBER = {},
			PAGES = {307--329},	
			}
		
		
		\bib{Masaki}{article}{
			AUTHOR = {Masaki, S.},
			TITLE = {Local existence and WKB approximation of
				solutions to Schr\"odinger-Poisson system in the two-dimensional whole
				space},
			JOURNAL = {Comm. P.D.E.},
			FJOURNAL = {Communications P.D.E.},
			VOLUME = {35},
			YEAR = {2010},
			NUMBER = {},
			PAGES = {2253-2278},
		}
		
		\bib{Masaki2}{article}{
			AUTHOR = {Masaki, S.},
			TITLE = {Energy Solution to a Schr\"odinger-Poisson System in the Two-Dimensional Whole Space},
			JOURNAL = {SIAM J. Math. Anal.},
			FJOURNAL = {SIAM J Math. Anal.},
			VOLUME = {43},
			YEAR = {2011},
			NUMBER = {6},
			PAGES = {2719-2731},
		}
		
		
		\bib{MSani}{article}{
			AUTHOR = {Masmoudi, N.},
			AUTHOR = {Sani, F.},
			TITLE = {Trudinger Moser Inequalities with the Exact Growth Condition in $\R^N$ ana Applications},
			JOURNAL = {Comm. Partial Differential Equations},
			FJOURNAL = {CPDE},
			VOLUME = {40},
			YEAR = {2015},
			NUMBER = {8},
			PAGES = {1408-1440},
		}
		
		
		
		
		\bib{moser}{article}{
			AUTHOR = {Moser, J.},
			TITLE = {A sharp form of an inequality by N. Trudinger},
			JOURNAL = {Ind. Univ. Math.},
			FJOURNAL = {Indiana University Mathematics},
			VOLUME = {30},
			YEAR = {1967},
			NUMBER = {},
			PAGES = {473--484},
		}
		
		
		\bib{ruf}{article}{
			AUTHOR = {Ruf, B.},
			TITLE = {A sharp Trudinger - Moser type inequality for
				unbounded domains in $\R^2$},
			JOURNAL = {J. Funct. Anal.},
			FJOURNAL = {Journal of Functional Analysis},
			VOLUME = {219},
			YEAR = {2005},
			NUMBER = {},
			PAGES = {340--367},
		}
		
		
		
		\bib{ruf2}{article}{
			AUTHOR = {Li,  Y.},
			AUTHOR={Ruf, B.}
			TITLE = {A sharp Trudinger - Moser type inequality for
				unbounded domains in $\R^N$},
			JOURNAL = {Indiana University Mathematics Journal},
			FJOURNAL = {},
			VOLUME = {57},
			YEAR = {2008},
			NUMBER = {1},
			PAGES = {3451--480},
		}
		
		\bib{Samko}{article}{
			AUTHOR = {Samko, S.},
			AUTHOR = {Kilbas, A.},
			AUTHOR = {Marichev, O.},
			TITLE = {Fractional Integrals and Derivatives },
			JOURNAL = {Gordon and Breach Science Publishers, Yverdon},
			FJOURNAL = {},
			VOLUME = {},
			YEAR = {1993},
			NUMBER = {},
			PAGES = {},
		}
		
		\bib{Stein}{article}{
			AUTHOR = {Stein, E.},
			TITLE = {Singular Integrals and Differentiability Properties of Functions},
			JOURNAL = { Princeton Mathematical Series},
			FJOURNAL = {Princeton University Press, Princeton, N.J.},
			VOLUME = {30},
			YEAR = {1970},
			NUMBER = {},
			PAGES = {},
		}
		
		\bib{Stinga}{article}{
			AUTHOR = {Stinga, P.R.},
			TITLE = {User's Guide to the Fractional Laplacian and the Method of Semigroups, Handbook of Fractional Calculus with Applications},
			JOURNAL = {De Gruyter, Berlin},
			FJOURNAL = {},
			VOLUME = {2},
			YEAR = {2019},
			NUMBER = {},
			PAGES = {235--265},	}
		
		
		
		\bib{stubbe}{article}{
			AUTHOR = {Stubbe, J.}
			TITLE = {Bound states of two-dimensional Schr\"odinger-Newton equations},
			JOURNAL={ArXive 0807.4059v1 (2008)}
		}
		
		
		\bib{SV}{article}{
			AUTHOR = {Stubbe, J., and M. Vuffray,}
			TITLE = {Bound states of the  Schr\"odinger- Newton model in low dimensions},
			JOURNAL = {Nonlinear Analysis},
			FJOURNAL = {Nonlinear Analysis},
			VOLUME = {73},
			YEAR = {2010},
			NUMBER = {},
			PAGES = {3171--3178},
		}
		
		
		
		
		\bib{Suzuki}{book}{
			AUTHOR = {Suzuki,T.},
			TITLE = {Free energy and self-interacting particles},
			BOOK = {Progress in Nonlinear Differential Equations and their Applications, Birkh\"auser, Boston},
			VOLUME = {62},
			YEAR = {2005},
		}
		
		
		
		
		%\bib{Strauss}{article}{
			%	AUTHOR = {Strauss, W.A.},
			%	TITLE = {Existence of solitary waves in higher dimensions},
			%	JOURNAL = {Comm. Math. Phys.},
			%	FJOURNAL = {Communications Mathematical Physics},
			%	VOLUME = {55},
			%	YEAR = {1977},
			%	NUMBER = {},
			%	PAGES = {149--162},
			%}
		
		
		\bib{trudinger}{article}{
			AUTHOR = {Trudinger, N.S.},
			TITLE = {On imbeddings into Orlicz spaces and some applications},
			JOURNAL = {J. Math. Mech.},
			FJOURNAL = {Indiana University Mathematics},
			VOLUME = {75},
			YEAR = {1980},
			NUMBER = {},
			PAGES = {},
		}
		
		\bib{W}{article}{
			AUTHOR = {Wolansky, G.},
			TITLE = {On steady distributions of self-attracting clusters under friction and fluctuations},
			JOURNAL = {Arch. Rational Mech. Anal.},
			FJOURNAL = {ARMA},
			VOLUME = {119},
			YEAR = {1992},
			NUMBER = {4},
			PAGES = {355-391},
		}
	\end{biblist}
	
	\end{bibdiv}
\end{document}